\documentclass[11pt, a4paper]{amsart}
\usepackage[utf8]{inputenc}
\usepackage{amsmath,amsthm,amssymb,amscd,faktor,float,mathdots,mathtools,enumerate,shuffle,hyperref,xcolor,enumitem,multicol,soul,parskip}
\usepackage{tikz, float, tikz-cd} 
	\usetikzlibrary{arrows}
\hypersetup{pdfborder={0 0 0}} 
\usepackage[
    a4paper,
    top=2.5cm,
    bottom=2.5cm,
    left=2.5cm,
    right=2.5cm,
    includeheadfoot,
    headheight=14pt,
    footskip=1cm,
    marginparwidth=3cm,
    marginparsep=0.5cm
]{geometry}

\newtheorem{theorem}{Theorem}[section]
\newtheorem{thm}{Theorem}

\newtheorem{proposition}[theorem]{Proposition}
\newtheorem{prop}[thm]{Proposition}
\newtheorem{propdef}[theorem]{Proposition-Definition}
\newtheorem{lemma}[theorem]{Lemma}
\newtheorem{corollary}[theorem]{Corollary}

\newtheorem{conj}[thm]{Conjecture}
\numberwithin{equation}{section}
\newtheorem*{theorem*}{Theorem}
\theoremstyle{definition}
\newtheorem{definition}[theorem]{Definition}

\newtheorem{remark}[theorem]{Remark}
\newtheorem*{acknowledgments}{Acknowledgments}

\newcommand{\Q}{\mathbb{Q}}

\newcommand{\MZV}{\mathcal{Z}}
\newcommand{\grMZV}{\operatorname{gr}_D\MZV}
\newcommand{\QX}{\Q\langle X\rangle}
\newcommand{\QY}{\Q\langle Y\rangle}
\newcommand{\ls}{\mathfrak{ls}}

\newcommand{\qMZV}{\mathcal{Z}_q}
\newcommand{\grqMZV}{\operatorname{gr}_D\mathcal{Z}_q}
\newcommand{\QB}{\Q\langle B\rangle}
\newcommand{\lqq}{\mathfrak{lq}}

\newcommand{\ad}{\operatorname{ad}}
\newcommand{\id}{\operatorname{id}}

\newcommand{\Lie}{\mathfrak{Lie}}
\newcommand{\stab}{\mathfrak{stab}}
\newcommand{\kernel}{\operatorname{ker}(\gamma_0)}

\newlist{steplist}{enumerate}{1}
\setlist[steplist]{label=\textbf{Step \arabic*}, align=left, wide=5pt, leftmargin=5pt}

\newlist{caselist}{enumerate}{1}
\setlist[caselist]{label=\textbf{Case \arabic*}, align=left, wide=5pt, leftmargin=5pt}

\title[A stabilizer interpretation of the linearized double shuffle Lie algebra]{A stabilizer interpretation of the (extended) linearized double shuffle Lie algebra}
\author{Annika Burmester}
\address{\scriptsize Faculty of Mathematics, Bielefeld University, 33615 Bielefeld, Germany.}
\email{\href{mailto:aburmester@math.uni-bielefeld.de}{aburmester@math.uni-bielefeld.de}}
\author{Khalef Yaddaden}
\address{\scriptsize Laboratoire Mathématique Nicolas Oresme (UMR 6139), University of Caen-Normandie, 6 Boulevard Maréchal Juin, 14000 Caen, France.}
\email{khalef.yaddaden@unicaen.fr}

\date{February 2026}

\subjclass[2020]{11M32, 
05A30, 
17B70
}
\keywords{multiple zeta values, q-analogues, linearized double shuffle relations, stabilizer Lie algebras.}

\begin{document}

\begin{abstract}
   The linearized double shuffle Lie algebra introduced by Brown reflects the depth-graded structure of multiple zeta values. In a previous paper, the first author introduced an extension of this Lie algebra that accommodates multiple q-zeta values and multiple Eisenstein series. Inspired by the stabilizer interpretation of the double shuffle Lie algebra given by Enriquez and Furusho, we provide in this paper a stabilizer interpretation of both Lie algebras and show that the stabilizers preserve the extension from the first linearized Lie algebra to the second one.
\end{abstract}

\maketitle

{\footnotesize \tableofcontents}

\section*{Introduction}

\textbf{Multiple zeta values.} Let $\MZV$ be the $\Q$-algebra of \emph{multiple zeta values}, so $\MZV$ is spanned by the elements
\begin{align*}
\zeta(k_1,\ldots,k_d)=\sum_{n_1>\cdots>n_d>0}\frac{1}{n_1^{k_1}\cdots n_d^{k_d}},\qquad k_1\geq2,\ k_2,\ldots,k_d\geq1.
\end{align*}
Here, $k_1+\cdots+k_d$ is the \emph{weight} and $d$ is the \emph{depth}.

A key feature of the algebra $\MZV$ is the two different product expressions. First, let $\QX$ be the non-commutative free algebra generated by the alphabet $X=\{x_0,x_1\}$. The \emph{shuffle product} on $\QX$ is given by $1\shuffle w=w\shuffle 1 =w$ and
\begin{align*}
av\shuffle bw=a(v\shuffle bw)+b(av\shuffle w),\qquad a,b\in X,\ v,w\in \QX.
\end{align*}
Let $\mathfrak{H}^0$ be the subspace of $\QX$, which is spanned by $1$ and all words starting in $x_0$ and ending in $x_1$.
\begin{prop} \label{prop:shuffle}
There is a surjective algebra morphism
\[
(\mathfrak{H}^0,\shuffle)\to\MZV,\quad x_0^{k_1-1}x_1\cdots x_0^{k_d-1}x_1\mapsto \zeta(k_1,\ldots,k_d),
\]
which has a unique extension to $\QX$ such that $x_0\mapsto 0$ and $x_1\mapsto 0$.
\end{prop}
Note that the shuffle product is homogeneous for the weight and the depth.
Next, consider the non-commutative free algebra $\QY$ generated by the alphabet $Y=\{y_1,y_2,\ldots\}$. The \emph{stuffle product} on $\QY$ is defined by $1\ast w=w\ast 1 =w$ and
\begin{align*}
y_iv\ast y_jw=y_i(v\ast y_jw)+y_j(y_iv\ast w)+y_{i+j}(v\ast w),\qquad i,j\geq1,\ v,w\in \QY.
\end{align*}
Let $\QY^0$ be the subspace of $\QY$ spanned by all words, which do not start in $y_1$.
\begin{prop} \label{prop:stuffle}
There is a surjective algebra morphism
\[
(\QY^0,\ast)\to\MZV,\quad y_{k_1}\cdots y_{k_d}\mapsto\zeta(k_1,\ldots,k_d),
\]
which has a unique extension to $\QY$ such that $y_1\mapsto0$.
\end{prop}
Note that the stuffle product is homogeneous for the weight and only filtered for the depth.
One of the central open problems in the field is the following well-known conjecture proposed by Ihara, Kaneko, and Zagier. 

\begin{conj} \label{con:IKZ}
All algebraic relations in the algebra $\MZV$ are obtained from the comparison of the extended maps $(\QX,\shuffle)\to \MZV$ and $(\QY,\ast)\to \MZV$ via the regularization map from \cite[Theorem 1]{IKZ}. 
\end{conj}

In particular, by Conjecture \ref{con:IKZ} the algebra $\MZV$ would be graded by weight. Moreover, the algebra $\MZV$ is equipped with an increasing depth filtration
\begin{align*}
\operatorname{Fil}_D^d\MZV := 
\operatorname{span}_\Q(\zeta(k_1,\ldots,k_r)\mid r\leq d),
\end{align*}
so
\[
\Q=\operatorname{Fil}_D^0\MZV\subset \operatorname{Fil}_D^1\MZV\subset \operatorname{Fil}_D^2\MZV\subset\cdots.
\]
We are interested in the associated depth-graded algebra
\begin{align*}
\grMZV := \Q\oplus\bigoplus_{d\geq1} \operatorname{Fil}_D^d\MZV / \operatorname{Fil}_D^{d-1}\MZV.
\end{align*}
Denote by $\zeta_D(k_1,\ldots,k_d)$ the equivalence class of $\zeta(k_1,\ldots,k_d)$ in $\grMZV$.

The shuffle product is weight homogeneous, hence by Proposition \ref{prop:shuffle} we get a surjective algebra morphism
\begin{align*}
(\QX,\shuffle)\to \grMZV,\quad x_0^{k_1-1}x_1\cdots x_0^{k_d-1}x_1\mapsto\zeta_D(k_1,\ldots,k_d).
\end{align*}
The stuffle product is only filtered for the depth, and the associated depth-graded product is the shuffle product on $\QY$, which we denote by $\shuffle_Y$. The map given in Proposition \ref{prop:stuffle} induces a surjective algebra morphism
\begin{align*}
(\QY,\shuffle_Y)\to \grMZV,\quad y_{k_1}\cdots y_{k_d}\mapsto\zeta_D(k_1,\ldots,k_d).
\end{align*}
Similar to Conjecture \ref{con:IKZ}, we expect the following.
\begin{conj}
All algebraic relations in the algebra $\grMZV$ are obtained by the comparison of the maps $(\QX,\shuffle)\to\grMZV$ and $(\QY,\shuffle_Y)\to\grMZV$.
\end{conj}
This conjecture motivates the definition of the \emph{linearized double shuffle Lie algebra} $\ls$ first given in \cite{Br21}, see Definition \ref{def:ls} and Theorem \ref{thm:ls_Lie}. 
\begin{conj} \label{conj:gr_DZ_and_ls}
There is an algebra isomorphism
\[
\grMZV/(\zeta_D(2))\simeq \mathcal{U}(\ls)^\vee.
\]
\end{conj}
In particular, the Lie algebra $\ls$ is expected to be the graded dual of the indecomposables of $\grMZV/(\zeta_D(2))$. 
Note that the Lie algebra $\ls$ is also studied from the perspective of its bigraded dual Lie coalgebra.
In fact, Maassarani proved in \cite{Mas} that this Lie coalgebra is isomorphic to the dihedral Lie coalgebra, 
studied by Goncharov in \cite{Gon,Gon2} in the context of multiple polylogarithms and mixed Tate motives.


In Section \ref{sec:ls}, we prove that $\ls$ is essentially the stabilizer of the bialgebra coproduct $\Delta_Y$, dual to the shuffle product $\shuffle_Y$ on $\QY$ for the action of the Lie algebra $\Lie(X)$ of primitive elements for the bialgebra coproduct $\Delta_X$, dual to the shuffle product $\shuffle$ on $\QX$. This approach is highly inspired from the work in \cite{EF18}. More precisely, we have the following main result:
\begin{thm}[Theorem \ref{thm:stab_equals_ls}] We have an isomorphism of bigraded spaces
\[
\stab_{\Lie(X)}(\Delta_Y)=\ls\oplus\Q x_0\oplus \Q x_1.
\]
\end{thm}
The explicit definition of the stabilizer $\stab_{\Lie(X)}(\Delta_Y)$ is given in Definition \ref{def:stab_Delta_Y}. We want to highlight the fact that the construction of the stabilizer induces a natural Lie algebra structure on $\stab_{\Lie(X)}(\Delta_Y)$. In particular, Theorem \ref{thm:stab_equals_ls} implies that the space $\ls$ is a Lie algebra, cf Theorem \ref{thm:ls_Lie}. This provides an alternative proof of a result originally stated in \cite[Theorem~5.5]{Br21}.

\vspace{0.5cm}
\textbf{Multiple q-zeta values.} For any parameter $0<q<1$, let $\qMZV$ be the $\Q$-algebra of multiple $q$-zeta values, which may be defined in several ways. A straightforward one is
\begin{align*}
\qMZV=\operatorname{span}_\Q\left\{\sum_{0<n_1<\cdots<n_r}\frac{q^{n_1s_1}}{(1-q^{n_1})^{s_1}}\cdots \frac{q^{n_rs_r}}{(1-q^{n_r})^{s_r}}\ \middle|\ s_1,\ldots,s_r\in\mathbb{Z}_{\geq0},\ s_r\geq1\right\}.
\end{align*}
For further details on the equivalent definitions and more background of the space $\qMZV$, we refer to \cite{BK20}, \cite{Bu23}. In the following, we focus on a specific spanning set of $\qMZV$ called the
\emph{balanced multiple q-zeta values} $\zeta_q(s_1,\ldots,s_r)$, where $s_1,\ldots,s_r\in\mathbb{Z}_{\geq0}$, $s_r\geq1$. The explicit construction is given in \cite{Bu24}, here we just explain their main properties. For $\zeta_q(s_1,\ldots,s_r)$, the \emph{weight} is $s_1+\cdots+s_r+\#\{s_i=0\}$ and the \emph{depth} is $r-\#\{s_i=0\}$.

Let $\QB$ be the non-commutative free algebra generated by the alphabet $B=\{b_0,b_1,b_2,\ldots\}$. The \emph{balanced quasi-shuffle product} on $\QB$ is defined by $1\ast_b w= w\ast_b 1=w$ and
\begin{align*}
b_iv\ast_b b_jw=b_i(v\ast_b b_jw)+b_j(b_iv\ast_b w)+\delta_{ij>0} b_{i+j}(v\ast_b w), \quad i,j\geq0,\ v,w\in\QB.
\end{align*}
Consider the subspace $\QB^0$ spanned by all words which do not end in $b_0$. The subspace $\QB^0$ is a subalgebra of $\QB$ freely generated by $(b_0^{m}b_k)_{m\geq0, k\geq1}$.

We define the algebra antiautomorphism $\tau:\QB^0\to \QB^0$ by
\begin{align*} 
\tau(b_0^mb_k)=b_0^{k-1}b_{m+1}
\end{align*}
for all $m\geq0$, $k\geq1$. The map $\tau$ is an involution.
\begin{prop} \label{prop:balanced_qsh_tau_inv}
There is a surjective, $\tau$-invariant algebra morphism
\[
(\QB^0,\ast_b)\to \qMZV,\quad b_{s_1}\cdots b_{s_r}\mapsto \zeta_q(s_1,\ldots,s_r),
\]
which has a unique extension to $\QB$ such that $b_0\mapsto0$.
\end{prop}
Note that the balanced quasi-shuffle product is homogeneous for the weight and only filtered for the depth. The map $\tau$ is homogeneous for the weight and depth.

As an analogue of Conjecture \ref{con:IKZ}, we have the following.
\begin{conj} \label{con:rel_in_qMZV}
All algebraic relations in $\qMZV$ are induced by the balanced quasi-shuffle product and the $\tau$-invariance of balanced multiple q-zeta values.
\end{conj}

In particular, Conjecture \ref{con:rel_in_qMZV} would imply that the algebra $\qMZV$ is graded by weight. Moreover, the algebra $\qMZV$ has an increasing depth filtration
\begin{align*}
\operatorname{Fil}_D^d\qMZV := \operatorname{span}_\Q(\zeta_q(s_1,\ldots,s_r)\mid r-\#\{s_i=0\}\leq d),
\end{align*}
and the associated depth-graded algebra is given by
\begin{align*}
\grqMZV := \Q\oplus\bigoplus_{d\geq1} \operatorname{Fil}_D^d\qMZV/\operatorname{Fil}_D^{d-1}\qMZV.
\end{align*}
We denote by $\zeta_{q,D}(s_1,\ldots,s_r)$ the equivalence class of $\zeta_q(s_1,\ldots,s_r)$ in $\grqMZV$.

The map $\tau$ is weight and depth homogeneous, hence by Proposition \ref{prop:balanced_qsh_tau_inv} we have $\tau$-invariance in $\grqMZV$, which is given by the following explicit identity
\begin{align*}
\zeta_{q,D}(\{0\}^{m_1},k_1,\ldots,\{0\}^{m_d},k_d)=\zeta_{q,D}(\{0\}^{k_d-1},m_d+1,\ldots,\{0\}^{k_1-1},m_1+1),
\end{align*}
for $k_j\geq1,\ m_j\geq0$. The balanced quasi-shuffle product is filtered for the depth, and the associated depth-graded product is the shuffle product on $\QB$, which we denote by $\shuffle_B$. So the map in Proposition \ref{prop:balanced_qsh_tau_inv} induces a surjective algebra morphism
\begin{align*}
(\QB,\shuffle_B)\to\grqMZV,\quad b_{s_1}\cdots b_{s_r}\mapsto \zeta_{q,D}(s_1,\ldots,s_r).
\end{align*}
These two properties motivate the definition of the \emph{linearized balanced Lie algebra} $\lqq$ introduced in \cite{Bu25}, see Definition \ref{def:lq} and Theorem \ref{thm:lq_Lie}. By construction, there is a surjective algebra morphism
\begin{align*} 
\mathcal{U}(\lqq)^\vee\twoheadrightarrow \grqMZV/\big(\zeta_{q,D}(2),\zeta_{q,D}(4),\zeta_{q,D}(6)\big).
\end{align*}
Note that in contrast to Conjecture \ref{conj:gr_DZ_and_ls} for multiple zeta values, we know that this map is not injective.

In Section \ref{sec:lq}, we prove that $\lqq$ is essentially the stabilizer of the involution $\tau$, for the action of the Lie algebra $\Lie(B)$ of primitive elements for the bialgebra coproduct $\Delta_B$, dual to the shuffle product $\shuffle_B$. The main result is as follows.
\begin{thm}[Theorem \ref{thm:stab_equal_lq}] We have an isomorphism of bigraded spaces
\[
\stab_{\Lie(B)}(\tau)=\lqq\oplus \Q b_0.
\]    
\end{thm}
For the explicit definition of the stabilizer $\stab_{\Lie(B)}(\tau)$ we refer to Definition \ref{def:stab_tau}. As before, the stabilizer $\stab_{\Lie(B)}(\tau)$ is naturally equipped with a Lie algebra structure. Therefore, by Theorem \ref{thm:stab_equal_lq} we obtain an alternative proof for $\lqq$ being a Lie algebra, cf Theorem \ref{thm:lq_Lie}. This was originally proven in \cite{Bu25}.

In Section \ref{sec:comparison} we relate the two stabilizers $\stab_{\Lie(X)}(\Delta_Y)$ and $\stab_{\Lie(B)}(\tau)$ studied in the previous sections. In fact, it has been proven in \cite[Theorem 7.10]{Bu25} that there exists an injective Lie algebra morphism $\theta : \ls \hookrightarrow \lqq$. In addition, the inclusions $\ls \subset \stab_{\Lie(X)}(\Delta_Y)$ and $\lqq \subset \stab_{\Lie(B)}(\tau)$ arising respectively from Theorems \ref{thm:stab_equals_ls} and \ref{thm:stab_equal_lq}, are Lie algebra inclusions. The following result shows that the map $\theta$ extends to the stabilizers.
\begin{thm}[Theorem \ref{thm:theta:stabDeltaY_to_stabtau}] The injective Lie algebra morphism $\theta : \ls \hookrightarrow \lqq$ induces an injective Lie algebra morphism
\begin{align*}
\theta^{(10)} : \big(\stab_{\Lie(X)}(\Delta_Y), \{-, -\}\big) \longrightarrow \big(\stab_{\Lie(B)}(\tau), \{-, -\}_A\big).
\end{align*}
\end{thm}
This result can be summarized by the fact that we have the following commutative diagram of injective Lie algebra morphisms:
\[\begin{tikzcd}
    \ls \ar[rr, hook, "\theta"] \ar[d, hook] && \lqq \ar[d, hook] \\
    \stab_{\Lie(X)}(\Delta_Y) \ar[rr, hook, "\theta^{(10)}"]&& \stab_{\Lie(B)}(\tau).
\end{tikzcd}\]
\vspace{0.3cm}
\begin{acknowledgments}
This project was partially supported by the first author's JSPS KAKENHI Grant 24KF0150 and Deutsche Forschungsgemeinschaft (DFG,German Research Foundation) – Project-ID 491392403 – TRR 358; and the second author's JSPS KAKENHI Grant 23KF0230. The authors are grateful to Benjamin Enriquez and Hidekazu Furusho for their fruitful comments.
\end{acknowledgments}

\section{The linearized double shuffle Lie algebra} \label{sec:ls}

\subsection{Algebraic setup}
Consider the non-commutative free $\Q$-algebra $\QX$ generated by $X := \{x_0,x_1\}$. It is equipped with a bialgebra structure with coproduct $\Delta_X : \QX \to \QX \otimes \QX$, which is the algebra morphism given by
\begin{equation} \label{eq:shuffle_coprod_X}
\Delta_X(x_i) = x_i \otimes 1 + 1 \otimes x_i, \quad \text{ for } i \in \{0, 1\}.
\end{equation}
Next, consider the non-commutative free $\Q$-algebra $\QY$ generated by $Y := \{y_1, y_2, \ldots \}$. It is equipped with a bialgebra structure with coproduct $\Delta_Y : \QY \to \QY \otimes \QY$, which is the algebra morphism given by
\begin{equation} \label{eq:shuffle_coprod_Y}
\Delta_Y(y_n) = y_n \otimes 1 + 1 \otimes y_n, \quad \text{ for } n \in \mathbb{Z}_{> 0}.
\end{equation}

The $\Q$-algebra morphism 
\begin{equation*}
    i_Y : \QY \to \QX, \quad y_n \mapsto x_0^{n-1} x_1
\end{equation*}
is injective. Thanks to that, we will - by abuse of notation - often identify $\QY$ with its image $i_Y(\QY)$ in $\QX$.
The direct sum decomposition of $\Q$-linear spaces
\[
    \QX = \QY\oplus \QX x_0
\]
induces a surjective $\Q$-linear map
\begin{equation*}
    \pi_Y: \QX \to \QY, \quad x_0^{k_1-1}x_1\cdots x_0^{k_d-1}x_1x_0^{r} \mapsto
    \begin{cases}
        y_{k_1}\cdots y_{k_d}, &\quad r=0, \\
        0 & \quad \text{otherwise}.
    \end{cases}
\end{equation*}
Therefore, we have a $\Q$-linear isomorphism 
\begin{equation} \label{eq:QY_simeq_quotient}
    \QY \simeq \QX / \QX x_0.
\end{equation}

Let $\Lie(X)$ be the free $\Q$-Lie algebra generated by the alphabet $X$. The $\Q$-algebra $\QX$ is isomorphic to the universal enveloping algebra of $\Lie(X)$. Therefore, $\Lie(X)$ is identified with the Lie algebra of primitive elements in $\QX$ for the coproduct $\Delta_X$ from \eqref{eq:shuffle_coprod_X}. Namely,
\[
    \Lie(X) = \{ \psi \in \QX \mid \Delta_X(\psi) = \psi \otimes 1 + 1 \otimes \psi \}.
\]
The Lie algebra $\Lie(X)$ is graded by weight, that is, it is equipped with a grading for which $x_0$ and $x_1$ are of degree $1$ and we have the decomposition
\[
    \Lie(X) = \bigoplus_{m \geq 1} \Lie(X)[m].
\]
The Lie algebra $\Lie(X)$ is also bigraded by weight and depth, that is, it is equipped with a bigrading for which $x_0$ is of bidegree $(1,0)$ and $x_1$ is of bidegree $(1,1)$ and we have the decomposition
\begin{equation} \label{eq:bigradecomp_LibX}
    \Lie(X) = \bigoplus_{m \geq 1, \ n \leq m} \Lie(X)[m,n].
\end{equation}

For $\psi \in \Lie(X)$, let $d_\psi$ be the derivation of $\QX$ given by
\begin{equation*} 
    d_{\psi}(x_0) = 0, \quad \text{ and } \quad d_\psi(x_1) = [x_1, \psi].
\end{equation*}
The triple $(\Lie(X),[-,-],d)$ is a post-Lie algebra, so by Proposition-Definition \ref{propdef:post-Lie} we get another Lie algebra bracket on $\Lie(X)$ given by
\begin{equation} \label{eq:Ihara_bracket}
\{\psi_1, \psi_2\} = d_{\psi_1}(\psi_2) - d_{\psi_2}(\psi_1) + [\psi_1, \psi_2].
\end{equation}
The Lie bracket \eqref{eq:Ihara_bracket} is usually referred to as the \emph{Ihara bracket}. 

\begin{lemma} \label{lem:Ihara_with_x0_x1}
The elements $x_0,x_1$ are central in the Lie algebra $(\Lie(X),\{-,-\})$.
\end{lemma}
\begin{proof} For any $\psi \in \Lie(X)$, we compute
\begin{align*} 
\{\psi,x_1\}=d_\psi(x_1)-d_{x_1}(\psi)+[\psi,x_1]=[x_1,\psi]+[\psi,x_1]=0.
\end{align*}
Moreover, we have
\begin{align*}
\{\psi,x_0\}=d_\psi(x_0)-d_{x_0}(\psi)+[\psi,x_0]=-d_{x_0}(\psi)+[\psi,x_0].
\end{align*}
Hence, we want to show that
\[-d_{x_0}(\psi)+[\psi,x_0]=0.\]
As $-d_{x_0}+[-,x_0]$ is a derivation on $\QX$, it suffices to check this equality for $\psi\in X$. We have
\begin{align*}
-d_{x_0}(x_1)+[x_1,x_0]&=-[x_1,x_0]+[x_1,x_0]=0, \\
-d_{x_0}(x_0)+[x_0,x_0]&=0.
\qedhere\end{align*}
\end{proof}

\begin{definition} \label{def:ls}
    The space $\ls$ consists of all $\psi \in \QX$ such that
    \begin{enumerate}[label=(\roman*), leftmargin=*, itemsep=5pt]
        \begin{multicols}{2}
        \item $(\psi\mid x_i)=0$ for $i=0,1$,
        \item $\Delta_X(\psi)=\psi\otimes1+1\otimes\psi$,
        \item $\Delta_Y(\pi_Y(\psi))=\pi_Y(\psi)\otimes1+1\otimes\pi_Y(\psi)$,
        \item $(\psi\mid x_0^{n-1}x_1)=0$ for $n\geq2$ even.
        \end{multicols}
    \end{enumerate}
\end{definition}

In \cite[Theorem 5.5]{Br21}, the following is stated without a detailed proof.
\begin{theorem} \label{thm:ls_Lie}
    The pair $(\ls,\{-,-\})$ is a $\Q$-Lie algebra.
\end{theorem}

\subsection{Stabilizer interpretation of the linearized double shuffle Lie algebra} 

In this section, our aim is to prove Theorem \ref{thm:ls_Lie} by showing that $\ls$ is essentially the stabilizer of the coproduct $\Delta_Y$ with respect to an action of the Lie algebra $(\Lie(X),\{-,-\})$.

For $\psi \in \Lie(X)$, let $s_\psi$ be the $\Q$-linear endomorphism of $\QX$ given by 
\[
    s_{\psi} := \ell_{\psi} + d_{\psi},
\]
where $\ell_\psi$ is the endomorphism of $\QX$ given by left multiplication by $\psi$.

\begin{lemma}[\cite{Ra02}, (3.1.9.2)] \label{lem:Lib_act_QX}
    There exists a $\Q$-Lie algebra action of $(\Lie(X), \{-,-\})$ by $\Q$-linear endomorphisms on $\QX$ given by
    \[
        (\Lie(X), \{-,-\}) \to \mathrm{End}_{\Q}(\QX), \quad \psi \mapsto s_{\psi}.
    \]
\end{lemma}

\begin{proof}
As $(\Lie(X),\{-,-\})$ is induced by a post-Lie structure, we can apply Proposition \ref{prop:generic_act_end} \ref{generic_act_end} with $\mathrm{end}_\psi = s_\psi$ for any $\psi \in \Lie(X)$ to obtain the above statement. 
\end{proof}
\noindent Notice that $s_\psi$ is an endomorphism of $\QX x_0$. Therefore, thanks to \eqref{eq:QY_simeq_quotient} it follows that there exists a unique $\Q$-linear endomorphism $s^Y_\psi$ of $\QY$ such that the following diagram
\[\begin{tikzcd}
    \QX \ar[rr, "s_\psi"] \ar[d, "\pi_Y"'] && \QX \ar[d, "\pi_Y"] \\
    \QY \ar[rr, "s^Y_\psi"] && \QY
\end{tikzcd}\]
commutes. Using this fact and Lemma \ref{lem:Lib_act_QX}, one shows the following.
\begin{proposition}[\cite{Ra02}, §4.1.1 and \cite{EF18}, Lemma 2.2] \label{prop:Lib_act_QY}
    There exists a $\Q$-Lie algebra action of $(\Lie(X), \{-,-\})$ by $\Q$-linear endomorphisms on $\QY$ given by
    \[
        (\Lie(X), \{-,-\}) \to \mathrm{End}_{\Q}(\QY), \quad \psi \mapsto s_\psi^Y.
    \]
\end{proposition}

\noindent Denote by $\mathrm{Cop}(\QY)$, the space of $\Q$-linear morphisms $\QY \to \QY^{\otimes 2}$. The action given in Proposition \ref{prop:Lib_act_QY} gives rise to an action of $(\Lie(X), \{-,-\})$ on the space $\mathrm{Cop}(\QX)$ by $\Q$-linear endomorphisms given by (\cite[§2.5]{EF18})
\begin{eqnarray}
    \label{LA_act_on_Delta*}
    (\Lie(X), \{-,-\}) & \longrightarrow & \mathrm{End}_{\Q}(\mathrm{Cop}(\QY)),
    \\
    \psi & \longmapsto & \Big( D \mapsto \left(s^Y_{\psi} \otimes \id + \id \otimes s^Y_{\psi}\right) \circ D - D \circ s^Y_{\psi}\Big). \notag
\end{eqnarray}

\begin{definition} \label{def:stab_Delta_Y}
    The stabilizer\footnote{The reader may refer to Proposition-Definition \ref{propdef:generic_stab} for a more general statement.} Lie algebra of $\Delta_Y \in \mathrm{Cop}(\QY)$ with respect to the action in \eqref{LA_act_on_Delta*} is the Lie subalgebra of $(\Lie(X), \{-,-\})$ given by
    \[
        \stab_{\Lie(X)}(\Delta_Y) := \left\{ \psi \in \Lie(X) \mid (s^Y_{\psi} \otimes \id + \id \otimes s^Y_{\psi}) \circ \Delta_Y = \Delta_Y \circ s^Y_\psi \right\}.
    \]
    In other words, $\stab_{\Lie(X)}(\Delta_Y)$ consists exactly of the elements $\psi \in \Lie(X)$ such that $s^Y_{\psi}$ is a coderivation of the bialgebra $(\QY, \Delta_Y)$.
\end{definition}

\begin{proposition} \label{prop:ls_subset_stab}
    We have the following inclusion of subspaces of $\Lie(X)$
    \[
        \mathfrak{ls} \subset \stab_{\Lie(X)}(\Delta_Y).
    \]
\end{proposition}

\noindent To prove this proposition, we first introduce some preliminary definitions and results. 
Let $D_f^Y$ be the $\Q$-linear endomorphism of $\QY$ given by 
\[
    D_f^Y := s_f^Y - r_{\pi_Y(f)},
\]
where $r_{\pi_Y(f)}$ denotes the right multiplication with $\pi_Y(f)$.
\begin{lemma}[{\cite[Proposition 4.1.4]{Ra02}}]
    The endomorphism $D_f^Y$ is a derivation of $\QY$.
\end{lemma}
In fact, thanks to \cite[Remark 4.1.5]{Ra02}, it is the restriction to $\QY$ of the derivation $D_f$ of $\QX$ given by
\[
    D_f := d_f + \mathrm{ad}_f,
\]
where one sets $\mathrm{ad}_f := [f, -]$.

Define $R_Y$ to be the algebra antiautomorphism of $\QY$ given by $y_n \mapsto y_n$ ($n \in \mathbb{Z}_{\geq 1}$). One checks that
\begin{equation} \label{eq:RY_DeltaY}
    \Delta_Y \circ R_Y = (R_Y \otimes R_Y) \circ \Delta_Y.
\end{equation}
Let $f$ be a homogeneous weight $p$ element of $\QX$. Such an element can be written as
\[
    f = \sum_{i=0}^p f_i x_0^i,
\]
where $f_i$ is a homogeneous weight $p-i$ element of $\QY$, for any $i \in \{0, \dots, p\}$. To $f$ one associates the sequence $(f_{i,j})_{i \in \{0, \dots, p-1\}, j \in \mathbb{Z}_{\geq 0}}$ given by
\[
    f_{i,0} = f_i + \overline{f}_i \text{ and } f_{i,j} = f_i y_j + y_j \overline{f}_i, \text{ for } j \geq 1,
\]
with $\overline{f}_i := (-1)^p R_Y(f_i)$.

Denote by $S_X$ the antipode of the Hopf algebra $(\QX, \Delta_X)$, that is, the algebra antiautomorphism of $\QX$ given by $x_i \mapsto - x_i$ ($i\in\{0,1\}$).
\begin{lemma}[{\cite[Lemma A.2]{Fu11}}] \label{lem:DYf_yn}
    Assume that $f$ satisfies $S_X(f) = - f$. Then, for $n \in \mathbb{Z}_{\geq 1}$, one has
    \[
        D_f^Y(y_n) = \sum_{i=0}^p f_{i,i+n}.
    \]
\end{lemma}

Let $\gamma_0$ be the derivation of $\QX$ given by $x_0 \mapsto 1$ and $x_1 \mapsto 0$. Define $\sec : \QY \to \QX$ to be the linear map given by 
\[
    w \mapsto \sum_{i \geq 0} \frac{(-1)^i}{i!} \gamma_0^i(w) x_0^i.
\]
One immediately checks that $\pi_Y \circ \sec = \id_{\QY}$. In fact, one also has the following.
\begin{lemma}[{\cite[Proposition 4.2.2]{Ra02}}] \label{lem:sec_piY}
    The map $\sec : \QY \to \ker(\gamma_0)$ is the inverse of ${\pi_Y}_{|\ker(\gamma_0)} : \ker(\gamma_0) \to \QY$.
\end{lemma}

\noindent One checks through a direct computation that
\[
    \gamma_0(y_1) = 0 \text{ and } \gamma_0(y_n) = (n-1) y_{n-1}, \text{ for } n \in \mathbb{Z}_{\geq 2}.
\]
The restriction of $\gamma_0$ to $\QY$ also defines a derivation that we will abusively denote $\gamma_0$ as well.
\begin{lemma} \label{lem:coder_partial0}
We have 
\[\Delta_Y \circ \gamma_0 = (\gamma_0 \otimes \id + \id \otimes \gamma_0) \circ \Delta_Y.\]
\end{lemma}
\begin{proof}
    It is enough to check this identity on the generators of $\QY$. This is immediate for $y_1$. For $n \geq 2$, we have
    \begin{align*}
        \Delta_Y \circ \gamma_0(y_n) & = (n-1) \Delta_Y(y_{n-1}) = (n-1) y_{n-1} \otimes 1 + 1 \otimes (n-1) y_{n-1} \\
        & = (\gamma_0 \otimes \id + \id \otimes \gamma_0)(y_n \otimes 1 + 1 \otimes y_n) = (\gamma_0 \otimes \id + \id \otimes \gamma_0) \circ \Delta_Y(y_n). \qedhere
    \end{align*}
\end{proof}

Let $\Lie(Y)$ be the free $\Q$-Lie algebra over the alphabet $Y$. In a similar way to $\Lie(X)$, the $\Q$-Lie algebra $\Lie(Y)$ is identified with the Lie subalgebra of primitive elements in $\QY$ for the coproduct $\Delta_Y$ from \eqref{eq:shuffle_coprod_Y}. Namely,
\[
    \Lie(Y) = \{ \phi \in \QY \mid \Delta_Y(\phi) = \phi \otimes 1 + 1 \otimes \phi \}.
\]

\begin{corollary} \label{cor:partial0_preserves_LibY}
    If $w \in \Lie(Y)$, then $\gamma_0(w) \in \Lie(Y)$.
\end{corollary}
\begin{proof}
    It follows immediately from Lemma \ref{lem:coder_partial0}, since $\Lie(Y)$ consists exactly of the primitive elements for $\Delta_Y$.
\end{proof}

For $n \in \mathbb{Z}_{\geq 1}$, let $\gamma_n$ be the derivation of $\QY$ defined by $y_k \mapsto \delta_{k,n}$, where $\delta_{k,n}$ is Kronecker's delta. The following result is somewhat written in \cite[Proposition 2.3.8]{Ra02}, but we find it useful to clarify the proof.
\begin{lemma} Let $n\in\mathbb{Z}_{\geq 1}.$
    \begin{enumerate}[leftmargin=*, label=(\alph*)]
        \item \label{gamma_n:LieYLieY} $\gamma_n([\Lie(Y), \Lie(Y)]) = 0$.
        \item \label{gamma_n:w} For any $w \in \Lie(Y)$, we have $\gamma_n(w) = (w | y_n)$.
    \end{enumerate}
    \label{lem:gamma_n}
\end{lemma}
\begin{proof}
    \begin{enumerate}[leftmargin=*, label=(\alph*)]
        \item Since $[\Lie(Y), \Lie(Y)]$ is a Lie subalgebra of $\Lie(Y)$ and $\gamma_n$ is a derivation, it is enough to check the desired identity on generators of $[\Lie(Y), \Lie(Y)]$. For $k,l \in \mathbb{Z}_{\geq 1}$, we have
        \[
            \gamma_n([y_k, y_l]) = [\gamma_n(y_k), y_l] + [y_k, \gamma_n(y_l)] = [\delta_{k,n}, y_l] + [y_k, \delta_{l,n}] = 0,   
        \]
        where the last equality comes from the fact that $1$ is central.
        \item Since $w \in \Lie(Y)$, it is a Lie polynomial on the $y_k$'s. It follows that $w - \sum_{k} (w | y_k) y_k$ is an element of $[\Lie(Y), \Lie(Y)]$. Therefore, thanks to \ref{gamma_n:LieYLieY}, we have
        \[
            \gamma_n\left(w - \sum_{k} (w | y_k) y_k\right) = 0.
        \]
        The result then follows by the linearity of $\gamma_n$ and the identity $\gamma_n(y_k) = \delta_{k,n}$.
    \end{enumerate}
\end{proof}

\begin{lemma} \label{lem:coder_difference}
    Let $g\in\Lie(Y)$ be a homogeneous weight $p$ element and set $f = \sec(g)$.
    \begin{enumerate}[label=(\alph*), leftmargin=*]
        \item \label{fi_LibY} For $i \in \{0, \dots, p\}$, we have
        \[
            f_i \in \Lie(Y) \text{ and } \overline{f}_i \in \Lie(Y).
        \]
    \end{enumerate}
    Moreover, if $S_X(f) = -f$, then
    \begin{enumerate}[label=(\alph*), leftmargin=*]
        \setcounter{enumi}{1}
        \item \label{f_i0} For $i \in \{0, \dots, p\}$, we have
        \[
            f_{i,0} = \big(1+(-1)^p\big)(-1)^{p-i} \binom{p-1}{i} (g | y_p) y_{p-i}.
        \]
        \item \label{coder_difference} For $n \in \mathbb{Z}_{\geq 1}$, we have
        \[
            \Delta_Y \circ D^Y_f(y_n) - (D_f^Y \otimes \id + \id \otimes D_f^Y) \circ \Delta_Y (y_n) = \sum_{i=0}^p f_{i,0} \otimes y_{i+n} + y_{i+n} \otimes f_{i,0}. 
        \]
    \end{enumerate}
\end{lemma}
\begin{proof}
    \begin{enumerate}[label=(\alph*), leftmargin=*]
        \item Let $i \in \{0, \dots, p\}$. Since $f = \sec(g)$, it follows that
        \[
            f_i = \frac{(-1)^i}{i!} \gamma_0^i(g).
        \]
        Since $g \in \Lie(Y)$, by Corollary \ref{cor:partial0_preserves_LibY}, it follows that $f_i \in \Lie(Y)$. Moreover, since $\overline{f}_i = (-1)^p R_Y(f_i)$ and since $R_Y$ satisfies identity \eqref{eq:RY_DeltaY}, it follows that $\overline{f}_i \in \Lie(Y)$.
        \item Let $i \in \{0, \dots, p\}$. Thanks to \ref{fi_LibY}, we have for any $k \in \{0, \dots, p\}$ that $f_k, \overline{f}_k \in \Lie(Y)$. Therefore, we get with Lemma  \ref{lem:gamma_n}\ref{gamma_n:w}
        \begin{align}
            \gamma_{i+1}\left(\sum_{k=0}^p f_{k,k+1}\right) & = \sum_{k=0}^p \left(\gamma_{i+1}(f_k) y_{k+1} + f_k \gamma_{i+1}(y_{k+1}) + \gamma_{i+1}(y_{k+1}) \overline{f}_k + y_{k+1} \gamma_{i+1}(\overline{f}_k)\right) \notag \\
            & = f_{i,0} + \sum_{k=0}^p (f_k + \overline{f}_k | y_{i+1}) y_{k+1} \label{partial_identity}  
        \end{align}
        Recall from the proof of \cite[Lemma A.4]{Fu11} that
        \[
            \sum_{k=0}^p (f_k + \overline{f}_k | y_{i+1}) y_{k+1} = \big(1+(-1)^p\big)(-1)^{p-i-1} \binom{p-1}{i} (g | y_p) y_{p-i}, 
        \]
        using the fact that, for any $k \in \{0, \dots, p\}$, $f_k$ is homogeneous of weight $p-k$, in addition to the identities $(\overline{f}_k | y_{i+1}) = (-1)^p (f_k | y_{i+1})$, and $f_{p-i-1} = \frac{(-1)^{p-i-1}}{(p-i-1)!} \gamma_0^{p-i-1}(g)$. \newline
        As $D_f^Y(y_1)=0$, we obtain from Lemma \ref{lem:DYf_yn} that
        \[
            \sum_{k=0}^p f_{k,k+1} = 0.
        \]
        Therefore, the left hand side of \eqref{partial_identity} is equal to zero, thus proving the claim.
        \item Since $S_X(f) = -f$, it follows from Lemma \ref{lem:DYf_yn} that $\displaystyle D_f^Y(y_n) = \sum_{i=0}^p f_{i,i+n}$. Therefore,
        \begin{align*}
            \Delta_Y \circ D^Y_f(y_n) & = \sum_{i=0}^p \Delta_Y(f_{i,i+n}) = \sum_{i=0}^p \Delta_Y(f_i) \Delta_Y(y_{i+n}) + \Delta_Y(y_{i+n}) \Delta_Y(\overline{f}_i) \\
            & = \mbox{\small$\displaystyle\sum_{i=0}^p (f_i \otimes 1 + 1 \otimes f_i) (y_{i+n} \otimes 1 + 1 \otimes y_{i+n}) + (y_{i+n} \otimes 1 + 1 \otimes y_{i+n}) (\overline{f}_i \otimes 1 + 1 \otimes \overline{f}_i)$} \\
            & = \sum_{i=0}^p (f_{i,i+n} \otimes 1 + 1 \otimes f_{i,i+n} + f_{i,0} \otimes y_{i+n} + y_{i+n} \otimes f_{i,0}),
        \end{align*}
        where the third equality follows from \ref{fi_LibY}. On the other hand, we have
        \begin{align*}
            (D_f^Y \otimes \id + \id \otimes D_f^Y) \circ \Delta_Y (y_n) & = (D_f^Y \otimes \id + \id \otimes D_f^Y)(y_n \otimes 1 + 1 \otimes y_n) \\
            & = D_f^Y(y_n) \otimes 1 + 1 \otimes D_f^Y(y_n) \\
            & = \sum_{i=0}^p f_{i,i+n} \otimes 1 + 1 \otimes f_{i,i+n}. 
        \end{align*}
        The difference $\Delta_Y \circ D^Y_f(y_n) - (D_f^Y \otimes \id + \id \otimes D_f^Y) \circ \Delta_Y (y_n)$ is then equal to the above formula.
    \end{enumerate}
\end{proof}

\begin{proof}[Proof of Proposition \ref{prop:ls_subset_stab}]
    Let $\psi \in \mathfrak{ls}$ be a homogeneous weight $p$ element. By definition we have that
    \[
        \mathfrak{ls} \subset \bigoplus_{m \geq 2} \Lie(X)[m] \subset \ker(\gamma_0).
    \]
    This implies by Lemma \ref{lem:sec_piY} that $\psi = \sec(\pi_Y(\psi))$. By definition of $\mathfrak{ls}$, we have $\pi_Y(\psi) \in \Lie(Y)$ and $\psi \in \Lie(X)$, so $S_X(\psi)=-\psi$. Therefore, assumptions of Lemma \ref{lem:coder_difference} are satisfied.
    Applying Lemma \ref{lem:coder_difference} to $g=\pi_Y(\psi)$ and $f=\psi$, it follows from \ref{f_i0} that
    \[
        \psi_{i,0} = \big(1+(-1)^p\big)(-1)^{p-i} \binom{p-1}{i} (\pi_Y(\psi) | y_p) y_{p-i} = 0,
    \]
    for any $i \in \{0, \dots, p\}$. This follows from the fact that in $\mathfrak{ls}$ we have
    \[
        (\pi_Y(\psi) | y_p) = (\psi | x_0^{p-1} x_1) = 0,
    \]
    for even $p$. For odd $p$, $\psi_{i,0}$ is always zero. Therefore, by Lemma \ref{lem:coder_difference} \ref{coder_difference}, it follows that
    \[
        \Delta_Y \circ D^Y_\psi(y_n) = (D_\psi^Y \otimes \id + \id \otimes D_\psi^Y) \circ \Delta_Y (y_n).
    \]
    Finally, since $\pi_Y(\psi) \in \Lie(Y)$, the right multiplication by $\pi_Y(\psi)$ is a coderivation of $(\QY, \Delta_Y)$. As $s_\psi^Y = D_\psi^Y + r_{\pi_Y(\psi)}$, it follows that $s_\psi^Y$ is a coderivation of $(\QY, \Delta_Y)$, thus proving the desired result. 
\end{proof}

\noindent Next, we aim to prove a somewhat converse with respect to the inclusion in Proposition \ref{prop:ls_subset_stab}. That is, we have the following result.
\begin{lemma} \label{lem:stab_subset_libY} 
    We have the following equality of subspaces of $\Lie(X)$
    \[
        \stab_{\Lie(X)}(\Delta_Y) \subset \Big(\Lie(X) \cap \pi_Y^{-1}(\Lie(Y))\Big) := \{ \psi \in \Lie(X) \mid \pi_Y(\psi) \in \Lie(Y) \}.
    \]
\end{lemma}
\begin{proof}
    Let $\psi \in \stab_{\Lie(X)}(\Delta_Y)$. In particular, it follows that
    \begin{equation} \label{eq:stab_to_1}
        (s^Y_{\psi} \otimes \id + \id \otimes s^Y_{\psi}) \circ \Delta_Y(1) = \Delta_Y \circ s^Y_\psi(1).
    \end{equation}
    Moreover, one checks that $s^Y_\psi(1) = \pi_Y(\psi)$. Therefore, equality \eqref{eq:stab_to_1} implies that
    \[
        \pi_Y(\psi) \otimes 1 + 1 \otimes \pi_Y(\psi) = \Delta_Y(\pi_Y(\psi)), 
    \]
    thus proving that $\pi_Y(\psi) \in \Lie(Y)$.
\end{proof}

\begin{theorem} \label{thm:stab_equals_ls}
    We have the following equality of bigraded subspaces of $\Lie(X)$
    \[
        \stab_{\Lie(X)}(\Delta_Y) = \mathfrak{ls} \oplus \Q x_0 \oplus \Q x_1.
    \]
\end{theorem}

In the proof of Theorem \ref{thm:stab_equals_ls}, we will use the following technical lemma.
\begin{lemma} \label{lem:nonvanishing_sum}
    For $m \in \mathbb{Z}_{\geq 1}$, the sum
    \[
        \sum_{k=0}^{m-1} (-1)^{k} \binom{m-1}{k} \left(y_{m-k} \otimes y_{k+2} + y_{k+2} \otimes y_{m-k} - y_{m-k+1} \otimes y_{k+1} - y_{k+1} \otimes y_{m-k+1} \right)
    \]
    vanishes if and only if $m$ is odd.
\end{lemma}
\begin{proof}
    First, one checks through direct computation that
    \begin{align*}
        & \sum_{k=0}^{m-1} (-1)^{k} \binom{m-1}{k} \left(y_{m-k} \otimes y_{k+2} + y_{k+2} \otimes y_{m-k} - y_{m-k+1} \otimes y_{k+1} - y_{k+1} \otimes y_{m-k+1} \right) \\
        & = - \sum_{k=0}^{m} (-1)^{k} \binom{m}{k} (y_{m-k+1} \otimes y_{k+1} + y_{k+1} \otimes y_{m-k+1}) \\
        & = - \sum_{k=0}^{m} (-1)^{m-k} \binom{m}{m-k} y_{k+1} \otimes y_{m-k+1} - \sum_{k=0}^{m} (-1)^{k} \binom{m}{k} y_{k+1} \otimes y_{m-k+1} \\
        & = - \sum_{k=0}^{m} \left((-1)^{m-k} + (-1)^k\right) \binom{m}{k} y_{k+1} \otimes y_{m-k+1}
    \end{align*}
    If $m$ is odd, then for any $k \in \{0, \dots, m\}$, $(-1)^{m-k} + (-1)^k = 0$. Thus making the sum vanish. Otherwise, if $m$ is even, then for any $k \in \{0, \dots, m\}$, $(-1)^{m-k} + (-1)^k = 2 (-1)^k$, thus the previous sum is equal to
    \[
        2 \sum_{k=0}^{m}  (-1)^{k-1} \binom{m}{k} y_{k+1} \otimes y_{m-k+1},
    \]
    which is not zero thanks to the freeness of the family $\left(y_{k+1} \otimes y_{m-k+1}\right)_{k \in \{0, \dots, m\}}$.
\end{proof}

\begin{proof}[Proof of Theorem \ref{thm:stab_equals_ls}]
    Depending of the pairs $(m,n)$, we may distinguish the following cases:
    \begin{caselist}
        \item $(m,n)=(1,0)$. Let us show that
        \[
            \stab_{\Lie(X)}(\Delta_Y)[1,0] = \mathfrak{ls}[1,0] \oplus \Q x_0.
        \]
        By definition of $\mathfrak{ls}$, it is immediate that $\mathfrak{ls}[1,0] = 0$. On the other hand, the decomposition \eqref{eq:bigradecomp_LibX} 
        implies the inclusion $\stab_{\Lie(X)}(\Delta_Y)[1,0] \subset \Q x_0$. The converse inclusion follows from the fact that $s_{x_0}^Y$ is a coderivation. Indeed, one notices that $s_{x_0}^Y = D_{x_0}^Y$. Therefore, $s_{x_0}^Y$ is a derivation and it is enough to check the identity
        \[
            (s^Y_{x_0} \otimes \id + \id \otimes s^Y_{x_0}) \circ \Delta_Y = \Delta_Y \circ s^Y_{x_0}
        \]
        on the generators $y_k$ of $\QY$, which can be verified through a straightforward computation that uses the fact that $s_{x_0}^Y(y_k)=0$ for any $k \in \mathbb{Z}_{\geq 1}$.
        \item $(m,n)=(1,1)$. Let us show that
        \[
            \stab_{\Lie(X)}(\Delta_Y)[1,1] = \mathfrak{ls}[1,1] \oplus \Q x_1.
        \]
        By definition of $\mathfrak{ls}$, it is immediate that $\mathfrak{ls}[1,1] = 0$. On the other hand, the decomposition \eqref{eq:bigradecomp_LibX} implies the inclusion $\stab_{\Lie(X)}(\Delta_Y)[1,1] \subset \Q x_1$. The converse inclusion follows from the fact that $s_{x_1}^Y$ is a coderivation. Indeed, one notices that $x_1=\sec(y_1)$ and since $S_X(x_1)=-x_1$,
        the identity of Lemma \ref{lem:coder_difference} \ref{coder_difference} holds for $f=x_1$. More precisely, in this identity we have $p=1$, $f_0=y_1$ and $f_1=0$. Therefore, $f_{0,0} = 0 = f_{1,0}$, which implies the equality
        \[
            \Delta_Y \circ D^Y_f(y_n) - (D_f^Y \otimes \id + \id \otimes D_f^Y) \circ \Delta_Y (y_n) = 0. 
        \]
        Thus, we established that $D_{x_1}^Y$ is a coderivation. Moreover, $\pi_Y(x_1) = y_1$ is primitive with respect to $\Delta_Y$, so $r_{\pi_Y(x_1)}$ also satisfies the coderivation property. Consequently, $s_{x_1}^Y = D_{x_1}^Y + r_{\pi_Y(x_1)}$ is itself a coderivation.
        \item $m \geq 2$ and $n \geq 2$. Let us show that
        \[
            \stab_{\Lie(X)}(\Delta_Y)[m, n] = \mathfrak{ls}[m, n].
        \]
        Indeed, decomposition \eqref{eq:bigradecomp_LibX} and the definition of 
        $\mathfrak{ls}$ imply that the bigraded linear spaces $\Lie(X) \cap \pi_Y^{-1}(\Lie(Y))$ and $\mathfrak{ls}$ coincide in bidegree $(m,n)$.
        The equality then follows thanks to the inclusions in Lemma \ref{lem:stab_subset_libY} and Proposition \ref{prop:ls_subset_stab}.
        \item $m \geq 2$ is odd and $n = 1$. Let us show that
        \[
            \stab_{\Lie(X)}(\Delta_Y)[m, 1] = \mathfrak{ls}[m, 1].
        \]
        Indeed, decomposition \eqref{eq:bigradecomp_LibX} and the definition of 
        $\mathfrak{ls}$ imply that the bigraded linear spaces $\Lie(X) \cap \pi_Y^{-1}(\Lie(Y))$ and $\mathfrak{ls}$ coincide in bidegree $(m,1)$.
        The equality then follows thanks to the inclusions in Lemma \ref{lem:stab_subset_libY} and Proposition \ref{prop:ls_subset_stab}.
        \item $m \geq 2$ is even and $n = 1$. Let us show that
        \[
            \stab_{\Lie(X)}(\Delta_Y)[m, 1] = \mathfrak{ls}[m, 1].
        \]
        By definition of $\mathfrak{ls}$, we have that $\mathfrak{ls}[m, 1]=0$. On the other hand, recall that
        \[
            \Lie(X)[m,1] = \Q \cdot \ad_{x_0}^{m-1}(x_1)
        \]
        and since $\pi_Y(\ad_{x_0}^{m-1}(x_1)) = y_m$ is $\Delta_Y$-primitive, it follows that the bigraded component of $\Lie(X) \cap \pi_Y^{-1}(\Lie(Y))$ of bidegree $(m,1)$ is equal to $\Q \cdot \ad_{x_0}^{m-1}(x_1)$, which contains $\stab_{\Lie(X)}(\Delta_Y)[m,1]$. One checks that
        \[
            s_{\ad_{x_0}^{m-1}(x_1)}^Y(1) = y_m \text{ and } s_{\ad_{x_0}^{m-1}(x_1)}^Y(y_2) = y_2 y_m + \sum_{k=0}^{m-1} (-1)^k \binom{m-1}{k} (y_{m-k} y_{k+2} - y_{m-k+1} y_{k+1}),
        \]
        where the last equality is obtained thanks to the identity
        \[
            \ad_{x_0}^{m-1}(x_1) = \sum_{k=0}^{m-1} (-1)^k \binom{m-1}{k} x_0^{m-k-1} x_1 x_0^k.
        \]
        Therefore,
        \begin{align}
            & \Delta_Y \circ s_{\ad_{x_0}^{m-1}(x_1)}^Y(y_2) = \Delta_Y\left(y_2 y_m + \sum_{k=0}^{m-1} (-1)^k \binom{m-1}{k} (y_{m-k} y_{k+2} - y_{m-k+1} y_{k+1})\right) \label{DeltaYsYy2} \\
            & = \left( y_{2} y_{m} + \sum_{k=0}^{m-1} (-1)^{k} \binom{m-1}{k} (y_{m-k} y_{k+2} - y_{m-k+1} y_{k+1}) \right) \otimes 1 \notag \\
            & + 1 \otimes \left( y_{2} y_{m} + \sum_{k=0}^{m-1} (-1)^{k} \binom{m-1}{k} (y_{m-k} y_{k+2} - y_{m-k+1} y_{k+1}) \right) + y_{2} \otimes y_{m} + y_{m} \otimes y_{2} \notag \\
            & + \sum_{k=0}^{m-1} (-1)^{k} \binom{m-1}{k} \left(y_{m-k} \otimes y_{k+2} + y_{k+2} \otimes y_{m-k} - y_{m-k+1} \otimes y_{k+1} - y_{k+1} \otimes y_{m-k+1} \right) \notag
        \end{align}
        and
        \begin{align}
            & \left(s_{\ad_{x_0}^{m-1}(x_1)}^Y \otimes \id + \id \otimes s_{\ad_{x_0}^{m-1}(x_1)}^Y\right) \circ \Delta_Y(y_2) \label{sYDeltaYy2} \\
            & = s_{\ad_{x_0}^{m-1}(x_1)}^Y(y_2) \otimes 1 + s_{\ad_{x_0}^{m-1}(x_1)}^Y(1) \otimes y_2 + y_2 \otimes s_{\ad_{x_0}^{m-1}(x_1)}^Y(1) + 1 \otimes s_{\ad_{x_0}^{m-1}(x_1)}^Y(y_2) \notag \\
            & = \left( y_{2} y_{m} + \sum_{k=0}^{m-1} (-1)^{k} \binom{m-1}{k} (y_{m-k} y_{k+2} - y_{m-k+1} y_{k+1}) \right) \otimes 1 + y_{m} \otimes y_{2} + y_{2} \otimes y_{m} \notag \\
            & + 1 \otimes \left( y_{2} y_{m} + \sum_{k=0}^{m-1} (-1)^{k} \binom{m-1}{k} (y_{m-k} y_{k+2} - y_{m-k+1} y_{k+1}) \right) \notag
        \end{align}
        Since $m$ is even, thanks to Lemma \ref{lem:nonvanishing_sum}, the last row of identity \eqref{DeltaYsYy2} does not vanish. This implies that the identities \eqref{DeltaYsYy2} and \eqref{sYDeltaYy2} are not equal. Therefore, $s_{\ad_{x_0}^{m-1}(x_1)}^Y$ is not a coderivation. Hence, $s_{\ad_{x_0}^{m-1}(x_1)} \notin \stab_{\Lie(X)}(\Delta_Y)[m, 1]$, implying that $\stab_{\Lie(X)}(\Delta_Y)[m, 1] = 0$.
    \end{caselist}
    Finally, since the bigraded subspaces $\stab_{\Lie(X)}(\Delta_Y)$ and $\mathfrak{ls} \oplus \Q x_0 \oplus \Q x_1$ of $\Lie(X)$ are equal on each bigraded component, the result follows.
\end{proof}




\begin{proof}[Proof of Theorem \ref{thm:ls_Lie}]
By Lemma \ref{lem:Ihara_with_x0_x1}, the space 
\[\stab_{\Lie(X)}(\Delta_Y)[1]=\stab_{\Lie(X)}(\Delta_Y)[1,0]\oplus \stab_{\Lie(X)}(\Delta_Y)[1,1]=\Q x_0 \oplus \Q x_1\] is contained in the center of the $\Q$-Lie algebra $(\Lie(X), \{-,-\})$. Therefore, the projection to the weight $1$ component
\begin{align} \label{eq:proj_stab(Delta_Y)_to_weight1}
\stab_{\Lie(X)}(\Delta_Y) \to \stab_{\Lie(X)}(\Delta_Y)[1]
\end{align}
is a $\Q$-Lie algebra morphism, where $\stab_{\Lie(X)}(\Delta_Y)[1]$ is an abelian $\Q$-Lie algebra. It follows that the kernel of the morphism in \eqref{eq:proj_stab(Delta_Y)_to_weight1} is an ideal of the $\Q$-Lie algebra $(\stab_{\Lie(X)}(\Delta_Y), \{-,-\})$. On the other hand, thanks to Theorem \ref{thm:stab_equals_ls}, we have that
\[
\stab_{\Lie(X)}(\Delta_Y) = \mathfrak{ls} \oplus \Q x_0 \oplus \Q x_1.
\]
This implies that
\[
\ker\left(\stab_{\Lie(X)}(\Delta_Y) \to \stab_{\Lie(X)}(\Delta_Y)[1]\right) = \mathfrak{ls}, 
\]
thus proving the result.
\end{proof}

\begin{corollary} We have the following inclusion of Lie subalgebras of $(\Lie(X),\{-,-\})$
\[
\{\stab_{\Lie(X)}(\Delta_Y),\stab_{\Lie(X)}(\Delta_Y)\}\subset \ls.
\]
\end{corollary}

\begin{proof} This follows directly from Theorem \ref{thm:stab_equals_ls} together with Lemma \ref{lem:Ihara_with_x0_x1}.
\end{proof}

\section{The linearized balanced Lie algebra} \label{sec:lq}

\subsection{Algebraic setup} Let $B := \{b_0, b_1, b_2, \ldots\}$ be an alphabet and $\QB$ be the non-commutative free $\Q$-algebra generated by $B$. It is equipped with a bialgebra structure, where the coproduct $\Delta_B:\QB\to\QB\otimes\QB$ is the algebra morphism defined by
\begin{align} \label{eq:def_Delta_B}
\Delta_B(b_i)= b_i\otimes 1+1\otimes b_i, \qquad i\in \mathbb{Z}_{\geq 0}.
\end{align}
Let $\QB^0$ be the subspace of $\QB$ spanned by words which do not end in $b_0$. The direct sum decomposition $\QB=\QB^0\oplus\QB b_0$ induces a canonical surjection
\begin{align*}
\pi_0:\QB\to \QB^0.
\end{align*}
Therefore, we have a $\Q$-linear isomorphism
\begin{align} \label{eq:QB^0_iso_QB/QB b_0}
\QB^0 \simeq \QB/\QB b_0.
\end{align}
Note that $\QB^0$ is also a subalgebra of $\QB$.

Let $\Lie(B)$ be the free $\Q$-Lie algebra generated by the alphabet $B$. The algebra $\QB$ is isomorphic to the universal enveloping algebra of $\Lie(B)$. So, $\Lie(B)$ is identified with the Lie algebra of primitive elements in $\QB$ for the coproduct $\Delta_B$ from \eqref{eq:def_Delta_B}. Namely,
\[
    \Lie(B) = \{ \psi \in \QB \mid \Delta_B(\psi) = \psi \otimes 1 + 1 \otimes \psi \}.
\]
The Lie algebra $\Lie(B)$ is graded for the weight, where $b_0$ has weight $1$ and for each $i\geq1$ the letter $b_i$ has weight $i$. We get a decomposition
\[
\Lie(B)=\bigoplus_{m\geq1} \Lie(B)[m],
\]
where $\Lie(B)[m]$ denotes the homogeneous subspace of weight $m$. The Lie algebra $\Lie(B)$ is also bigraded for the weight and depth, where the letter $b_0$ is of bidegree $(1,0)$ and for $i\geq1$ the letter $b_i$ is of bidegree $(i,1)$. As before, we get a decomposition
\[
\Lie(B)=\bigoplus_{m\geq,\ n\leq m} \Lie(B)[m,n],
\]
where $\Lie(B)[m,n]$ denotes the homogeneous subspace of bidegree $(m,n)$.

For a word $w=b_0^{m_1}b_{k_1}\cdots b_0^{m_d}b_{k_d}b_0^{m_{d+1}}$ in $\QB$, where $k_1,\ldots,k_d\geq1,\ m_1,\ldots,m_{d+1}\geq0$, and an index $\mathbf{l}=(l_1,\ldots,l_d)\in \mathbb{Z}_{>0}^d$, we use the notation
\begin{align*}
w(\mathbf{l})=b_0^{m_1}b_{l_1}\cdots b_0^{m_d}b_{l_d}b_0^{m_{d+1}}.
\end{align*}
In particular, if $\mathbf{k}=(k_1,\ldots,k_d)$, then $w(\mathbf{k})=w$. For $\mathbf{k}=(k_1,\ldots,k_d),\ \mathbf{l}=(l_1,\ldots,l_d)$, we write $\mathbf{l}\leq\mathbf{k}$ if $l_i\leq k_i$ for $i\in\{1,\ldots,d\}$. Moreover, we abbreviate
\[|\mathbf{k}|=k_1+\cdots+k_d,\qquad \binom{\mathbf{k}-1}{\mathbf{l}-1}=\prod_{i=1}^d \binom{k_i-1}{l_i-1}. \]

For $w=b_0^{m_1}b_{k_1}\cdots b_0^{m_d}b_{k_d}b_0^{m_{d+1}} \in \QB$, define the derivation $\partial_w:\QB\to\QB$ by
\begin{align*} 
\partial_w(b_0)&=0, \\ 
\partial_w(b_i)&=\sum_{\mathbf{l}\leq\mathbf{k}} (-1)^{|\mathbf{k}|+|\mathbf{l}|} \binom{\mathbf{k}-1}{\mathbf{l}-1} [b_{i+|\mathbf{k}|-|\mathbf{l}|},w(\mathbf{l})], \qquad i\geq1.
\end{align*}
If $w=b_0^m$ for some $m\geq1$, then we set
\[
\partial_w(b_0)=0,\qquad \partial_w(b_i)=[b_i,w].
\]
We define
\begin{equation*} 
\{\psi_1,\psi_2\}_A=\partial_{\psi_1}(\psi_2)-\partial_{\psi_2}(\psi_1)+[\psi_1,\psi_2],\quad \psi_1,\psi_2\in \Lie(B).
\end{equation*}
An immediate consequence of \cite[Theorem 4.13]{BK25} is the following.
\begin{proposition} 
For $\psi_1, \psi_2 \in \QB$, we have the following equality in $\mathrm{Der}_{\Q}(\QB)$
\[
    \partial_{\{\psi_1, \psi_2\}_A}=[\partial_{\psi_1}, \partial_{\psi_2}].
\]    
\end{proposition}
Thus, one may apply Proposition-Definition \ref{propdef:post-Lie} to $\mathfrak{g}=\Lie(B)$ and $\psi_1 \triangleright \psi_2 = \partial_{\psi_1}(\psi_2)$, and hence deduce that $(\Lie(B),\{-,-\}_A)$ is a Lie algebra.
\begin{remark} \label{rem:bracket_on_b0_b1} The restriction of the derivation $\partial$ to $\Q\langle b_0,b_1\rangle$ equals the derivation corresponding to the Ihara bracket, precisely, we have for $\psi\in \Lie(b_0,b_1)$
\[
\partial_\psi(b_0)=0, \qquad \partial_\psi(b_1)=[b_1,\psi].
\]
In particular, the restriction of the Lie bracket $\{-,-\}_A$ to $\Lie(b_0,b_1)$ is just the Ihara bracket.
\end{remark}
\begin{lemma} \label{lem:bracket_b1b2..}
The Lie bracket $\{-,-\}_A$ preserves the subspace $\Lie(b_1,b_2,\ldots)\subset \Lie(B)$.
\end{lemma}
\begin{proof} This follows from the observation that the derivation $\partial$ and hence also the Lie bracket $\{-,-\}_A$ does not change the number of $b_0$. 
\end{proof}

\begin{lemma} \label{lem:Ari_with_b0}
The element $b_0$ is central in the Lie algebra $(\Lie(B),\{-,-\}_A)$.
\end{lemma}

\begin{proof} For any $\phi\in\QB$, we compute
\begin{align*}
\{\phi,b_0\}=\partial_\phi(b_0)-\partial_{b_0}(\phi)+[\phi,b_0]=-\partial_{b_0}(\phi)+[\phi,b_0].
\end{align*}
Thus, it suffices to verify that
\[-\partial_{b_0}(\phi)+[\phi,b_0]=0.\]
The map $-\partial_{b_0}+[-,b_0]$ is a derivation on $\QB$, hence we only have to prove this equality for $\phi \in B$. We compute for $i\geq 1$
\begin{align*}
-\partial_{b_0}(b_i)+[b_i,b_0]&=-[b_i,b_0]+[b_i,b_0]=0, \\
-\partial_{b_0}(b_0)+[b_0,b_0]&=0.
\qedhere\end{align*}
\end{proof}

\begin{definition} \label{def:lq}
Let $\lqq$ be the set of all $\psi\in \QB$ such that
\begin{enumerate}[label=(\roman*), leftmargin=*, itemsep=5pt]
\begin{multicols}{2}
\item $(\psi\mid b_0)=0$,
\item $\Delta_B(\psi)=\psi\otimes1+1\otimes\psi$,
\item $\tau(\pi_0(\psi))=\pi_0(\psi)$,
\item $(\psi\mid b_0^mb_k)=0$ for $k+m$ even,
\end{multicols}
\end{enumerate}
where $\tau$ is the algebra antiautomorphism of $\QB^0$ given by
\begin{align*} 
\tau(b_0^mb_k)=b_0^{k-1}b_{m+1}
\end{align*}
for any integers $m\geq0$, $k\geq1$.
\end{definition}

\begin{theorem} \label{thm:lq_Lie} \cite[Theorem 4.3]{Bu25}
The pair $(\lqq,\{-,-\}_A)$ is a $\Q$-Lie algebra.
\end{theorem}

\subsection{Stabilizer interpretation of the linearized balanced Lie algebra}
We give an alternative proof for Theorem \ref{thm:lq_Lie} by showing that $\lqq$ is essentially the stabilizer of the involution $\tau$ with respect to an action of the Lie algebra $(\Lie(B),\{-,-\}_A)$.

For $\psi \in \Lie(B)$, let $\sigma_\psi$ be the $\Q$-linear endomorphism of $\QB$ given by
\begin{equation} \label{eq:def_sigma}
    \sigma_\psi := \ell_\psi + \partial_{\psi}.
\end{equation}
where $\ell_\psi$ is the endomorphism of $\QB$ given by left multiplication by $\psi$.

\begin{lemma} \label{lem:Lib_act_QB}
There exists a $\Q$-Lie algebra action of $(\Lie(B), \{-,-\}_A)$ by $\Q$-linear endomorphisms on $\QB$ given by
\[
(\Lie(B), \{-,-\}_A) \to \mathrm{End}_{\Q}(\QB), \quad \psi \mapsto \sigma_{\psi}.
\]
\end{lemma}
\begin{proof}
As $(\Lie(B),\{-,-\}_A)$ is induced from a post-Lie structure, we can apply Proposition \ref{prop:generic_act_end} \ref{generic_act_end} with $\mathrm{end}_\psi = \sigma_\psi$ for any $\psi \in \Lie(B)$. 
\end{proof}

Note that $\sigma_\psi$ is an endomorphism of $\QB b_0$. Therefore, thanks to \eqref{eq:QB^0_iso_QB/QB b_0} we obtain a unique $\Q$-linear endomorphism $\sigma^0_\psi$ of $\QB^0$ such that the diagram
\begin{equation} \label{eq:diagram_sigma_0}
\begin{tikzcd}
    \QB \ar[rr, "\sigma_\psi"] \ar[d, "\pi_0"'] && \QB \ar[d, "\pi_0"] \\
    \QB^0 \ar[rr, "\sigma^0_\psi"] && \QB^0
\end{tikzcd}
\end{equation}
commutes.
\begin{proposition} \label{prop:Lib_act_QB0}
    There exists a $\Q$-Lie algebra action of $(\Lie(B), \{-,-\}_A)$ by $\Q$-linear endomorphisms on $\QB^0$ given by
    \[
        (\Lie(B), \{-,-\}_A) \to \mathrm{End}_{\Q}(\QB^0), \quad \psi \mapsto \sigma_\psi^0.
    \]
\end{proposition}

\begin{proof} This follows from Lemma \ref{lem:Lib_act_QB} and the diagram in \eqref{eq:diagram_sigma_0}.
\end{proof}

\begin{lemma} The action given in Proposition \ref{prop:Lib_act_QB0} gives rise to an action of $(\Lie(B), \{-,-\}_A)$ on the space $\mathrm{End}(\QB^0)$ of endomorphisms given by 
\begin{equation}
    \label{LA_act_on_tau}
    \psi \longmapsto \left( t \mapsto \sigma^0_\psi \circ t - t \circ \sigma^0_\psi \right),
\end{equation}
where $\psi \in \Lie(B)$ and $t \in \mathrm{End}(\QB^0)$. 
\end{lemma}

\begin{proof}
Set $\psi\cdot t:=\sigma_\psi^0\circ t - t\circ \sigma_\psi^0$ for any $\psi\in\Lie(B)$ and $t\in \mathrm{End}(\QB^0)$. \newline
For $\phi,\psi\in\Lie(B)$, and $t\in\mathrm{End}(\QB^0)$, we have
\begin{align*}
\phi\cdot (\psi\cdot t) - \psi\cdot(\phi\cdot t)&=\phi\cdot (\sigma_\psi^0\circ t - t\circ \sigma_\psi^0)-\psi\cdot(\sigma_\phi^0\circ t - t\circ \sigma_\phi^0) \\
&=\sigma_\phi^0\circ \sigma_\psi^0 \circ t - \sigma_\phi^0\circ t\circ \sigma_\psi^0 -\sigma_\psi^0\circ t \circ \sigma_\phi^0 + t\circ \sigma_\psi^0\circ \sigma_\phi^0 \\
&\hspace{0.4cm} - \sigma_\psi^0\circ \sigma_\phi^0\circ t +\sigma_\psi^0\circ t\circ \sigma_\phi^0 +\sigma_\phi^0\circ t \circ \sigma_\psi^0 - t\circ \sigma_\phi^0\circ \sigma\psi^0 \\
&=[\sigma_\phi^0,\sigma_\psi^0]\circ t - t\circ [\sigma_\phi^0,\sigma_\psi^0] \\
&=\pi_0\circ [\sigma_\phi,\sigma_\psi]\circ t - t\circ \pi_0\circ [\sigma_\phi,\sigma_\psi].
\end{align*}
By Lemma \ref{lem:Lib_act_QB}, we have $[\sigma_\phi,\sigma_\psi]=\sigma_{\{\phi,\psi\}_A}$. Hence, we deduce
\begin{align*}
\phi\cdot (\psi\cdot t) - \psi\cdot(\phi\cdot t)&=\pi_0\circ\sigma_{\{\phi,\psi\}_A}\circ t - t\circ \pi_0\circ \sigma_{\{\phi,\psi\}_A}\\
&=\sigma_{\{\phi,\psi\}_A}^0\circ t - t\circ \sigma_{\{\phi,\psi\}_A}^0 \\
&=\{\phi,\psi\}_A\cdot t. 
\end{align*}
Thus, the map in \eqref{LA_act_on_tau} is a Lie algebra morphism.
\end{proof}

\begin{definition} \label{def:stab_tau}
The stabilizer\footnote{The reader may refer to Proposition-Definition \ref{propdef:generic_stab} for a more general statement.} Lie algebra of $\tau \in \mathrm{End}(\QB^0)$ with respect to the action in \eqref{LA_act_on_tau} is the Lie subalgebra of $(\Lie(B), \{-,-\}_A)$ given by
\[
    \stab_{\Lie(B)}(\tau) := \left\{ \psi \in \Lie(B) \mid \sigma^0_\psi \circ \tau = \tau \circ \sigma^0_\psi \right\}.
\]
\end{definition}
In other words, $\stab_{\Lie(B)}(\tau)$ consists exactly of the elements $\psi\in \Lie(B)$ such that $\sigma_\psi^0$ commutes with $\tau$.
By construction, $\stab_{\Lie(B)}(\tau)$ is a Lie subalgebra of $(\Lie(B), \{-,-\}_A)$.

\begin{proposition} \label{prop:lq_subset_stab}
We have the following inclusion of subspaces of $\Lie(B)$
\[
\lqq\subset \stab_{\Lie(B)}(\tau).
\]
\end{proposition}

To prove this proposition, we first introduce several maps and results.
For a word $w=b_0^{m_0}b_{k_1}\cdots b_0^{m_d}b_{k_d}$ in $\QB^0$, where $k_1,\ldots,k_d\geq1,\ m_1,\ldots,m_d\geq0$, and an index $\mathbf{n}=(n_1,\ldots,n_d)\in \mathbb{Z}_{\geq0}^d$ we use the notation
\begin{align*}
w(\overline{\mathbf{n}})=b_0^{n_1}b_{k_1}\cdots b_0^{n_d}b_{k_d}.
\end{align*}
In particular, if $\mathbf{m}=(m_1,\ldots,m_d)$ then $w(\overline{\mathbf{m}})=w$.
The map $\sec:\QB^0\to \QB$ is defined by 
\begin{align*}
\sec(w)=\sum_{\mathbf{n}\leq \mathbf{m}}(-1)^{|\mathbf{m}|+|\mathbf{n}|}\binom{\mathbf{m}}{\mathbf{n}} w(\overline{\mathbf{n}})b_0^{|\mathbf{m}|-|\mathbf{n}|}, \text{ for } w = w(\overline{\mathbf{m}}) \in \QB^0. 
\end{align*}
Denote by $\gamma_0:\QB\to\QB$ the derivation given by $\gamma_0(b_0)=1$ and $\gamma_0(b_a)=0$ for $a\geq1$. Then, we have
\begin{align*}
\kernel=\Q\langle \ad_{b_0}^m(b_k)\mid k\geq1,m\geq0\rangle.
\end{align*}

\begin{proposition} \label{prop:sec_pi0_inverse} \cite[Proposition 5.2]{Bu25} We have
\begin{enumerate}[leftmargin=*, label=(\alph*)]
    \item $\pi_0 \circ \sec = \id_{\QB^0}$,
    \item $\sec \circ \pi_0 = \id_{\kernel}$.
\end{enumerate}
\end{proposition}

Let $\rho : \QB^0 \to \QB^0$ be the $\Q$-linear endomorphism given by
\[\rho(b_0^{m_1}b_{k_1}\cdots b_0^{m_d}b_{k_d})=\hspace{-0.3cm}\sum_{\substack{l_1+\cdots+l_d=k_1+\cdots+k_d \\ n_1+\cdots+n_d=m_1+\cdots+m_d \\ l_s\geq1,\ n_s\geq0}} \hspace{-0.1cm} (-1)^{l_d+n_d-1}\prod_{s=1}^{d-1} \binom{k_s-1}{l_s-1}\binom{m_s}{n_s} b_0^{n_d}b_{l_d}b_0^{n_1}b_{l_1}\cdots b_0^{n_{d-1}}b_{l_{d-1}}.\]
The map $\rho$ is essentially a composition of $\tau$ and the antipode
\[S:\QB\to \QB,\quad b_{s_1}\cdots b_{s_r}\mapsto (-1)^r b_{s_r}\cdots b_{s_1}\]
of the Hopf algebra $(\QB,\Delta_B)$. Precisely, consider the composition $S_0=\pi_0\circ S\circ \sec$. Then, for any $w\in \QB^0$ we have
\begin{align*}
\rho(w)=(-1)^{\operatorname{wt}(w)+\operatorname{dep}(w)} \big(S_0\circ\tau\circ S_0\circ \tau\big)(w).
\end{align*}
\begin{proposition} \label{prop:lq_rho_inv} \cite[Proposition 5.4]{Bu25}
If $\psi\in\lqq$, then we have 
\[\rho(\pi_0(\psi))=\pi_0(\psi).\]
\end{proposition}

For $w=b_0^{m_1}b_{k_1}\cdots b_0^{m_d}b_{k_d}b_0^{m_{d+1}}\in\QB$, we set
\begin{align*}
\partial_w^R(b_0)&=\partial_w^L(b_0)=0,\\
\partial_w^R(b_a)&=\sum_{\mathbf{l}\leq \mathbf{k}} (-1)^{|\mathbf{k}|+|\mathbf{l}|}\binom{\mathbf{k}-1}{\mathbf{l}-1} b_{a+|\mathbf{k}|-|\mathbf{l}|}w(\mathbf{l}), \\
\partial_w^L(b_a)&=\sum_{\mathbf{l}\leq \mathbf{k}} (-1)^{|\mathbf{k}|+|\mathbf{l}|}\binom{\mathbf{k}-1}{\mathbf{l}-1} w(\mathbf{l})b_{a+|\mathbf{k}|-|\mathbf{l}|}.
\end{align*}
In particular, we have
\begin{align} \label{eq:partial_LR}
\partial_w=\partial_w^R-\partial_w^L. 
\end{align}
Moreover, we define for any $w\in\QB^0$ the map $\partial_{w}^0:\QB^0\to\QB^0$ by
\begin{align} \label{eq:def_partial0}
\partial_{w}^0(\pi_0(v))=\pi_0(\partial_{\operatorname{sec}(w)}(v)), \qquad v\in \QB.
\end{align}
Similarly, we define maps $\partial_w^{R,0},\partial_w^{L,0}:\QB^0\to\QB^0$.
\begin{proposition} \label{prop:tau_partial} \cite[Proposition 5.5]{Bu25} For non-empty words $v,w\in\QB^0$, we have
\begin{enumerate}[leftmargin=*, label=(\alph*)]
\item $\big(\tau\circ \partial_{\tau(w)}^{R,0}\circ\tau\big)(v)= \partial_w^{R,0}(v)+\operatorname{sec}(w)v-\tau\big(\operatorname{sec}(\tau(w))\tau(v)\big)$,
\item $\big(\tau\circ \partial_{\tau(w)}^{L,0}\circ \tau\big)(v)= \partial_{\rho(w)}^{L,0}(v)$.
\end{enumerate}
\end{proposition}

\begin{proof}[Proof of Proposition \ref{prop:lq_subset_stab}]
We prove that for $\psi \in \lqq$ the following equality of $\Q$-linear maps from $\QB$ to $\QB^0$
\begin{align*} 
\tau \circ \sigma_\psi^0 \circ \tau \circ \pi_0 = \pi_0 \circ \sigma_\psi.
\end{align*}
Thanks to the surjectivity of $\pi_0$, this directly implies $\sigma_\psi^0 \circ \tau = \tau \circ \sigma_\psi^0$ and hence the claimed inclusion $\lqq \subset \stab_{\Lie(B)}(\tau)$. If $v \in \QB b_0$, then
\[
    \tau \circ \sigma_\psi^0 \circ \tau \circ \pi_0(v) = 0 = \pi_0 \circ \sigma_\psi(v).
\]
Let $v$ be a non-empty word in $\QB^0$. As $\tau\circ \pi_0=\pi_0\circ \tau$ on $\QB^0$ and by definition of $\sigma_\psi ^0$ given in \eqref{eq:diagram_sigma_0}, we have
\begin{align*}
\big(\tau\circ \sigma_\psi^0\circ\tau\circ\pi_0\big)(v)&=\big(\tau\circ \sigma_\psi^0\circ\pi_0\circ\tau\big)(v)=\big(\tau\circ\pi_0\circ\sigma_\psi\circ\tau\big)(v).
\end{align*}
Thus, we get with \eqref{eq:def_sigma} and \eqref{eq:partial_LR}
\begin{align*}
&\big(\tau\circ \sigma_\psi^0\circ\tau\circ\pi_0\big)(v)=(\tau\circ\pi_0\circ\ell_\psi)(\tau(v))+\big(\tau\circ\pi_0\circ \partial_\psi^R\big)(\tau(v))-\big(\tau\circ \pi_0\circ \partial_\psi^L\big)(\tau(v)).
\end{align*}
By Proposition \ref{prop:sec_pi0_inverse}, we have for $\psi\in\lqq$ that $\sec(\pi_0(\psi))=\psi$. So by definition of $\partial_\psi^{R,0}$, $\partial_\psi^{L,0}$ as in \eqref{eq:def_partial0} we get
\begin{align*}
\big(\tau\circ \sigma_\psi^0\circ\tau\circ\pi_0\big)(v)&=\big(\tau\circ\pi_0\big)(\psi\tau(v))+\big(\tau\circ\partial_{\pi_0(\psi)}^{R,0}\big)(\tau(v))-\big(\tau\circ \partial_{\pi_0(\psi)}^{L,0}\big)(\tau(v))\\
&=\tau\big(\psi\tau(v)\big)+\partial_{\pi_0(\psi)}^{R,0}(v)+\sec(\pi_0(\psi))v-\tau\big(\sec(\pi_0(\psi))\tau(v)\big)-\partial_{\rho(\pi_0(\psi))}^{L,0}(v),
\end{align*}
where the second equality follows from Proposition \ref{prop:tau_partial}. Applying again Proposition \ref{prop:sec_pi0_inverse} and using the $\rho$-invariance of $\pi_0(\psi)$ (cf Proposition \ref{prop:lq_rho_inv}), we obtain
\begin{align*}
\big(\tau\circ \sigma_\psi^0\circ\tau\circ\pi_0\big)(v)&=\tau\big(\psi\tau(v)\big)+\partial_{\pi_0(\psi)}^{R,0}(v)+\psi v-\tau\big(\psi\tau(v)\big)-\partial_{\pi_0(\psi)}^{L,0}(v) \\
&=\partial_{\pi_0(\psi)}^0(v)+\ell_\psi(v)\\
&=\big(\pi_0\circ \partial_\psi\big)(v)+\ell_{\psi}(v)\\
&=(\pi_0\circ\sigma_\psi\big)(v). \qedhere
\end{align*}
\end{proof}

The next result is nearly the converse statement of Proposition \ref{prop:lq_subset_stab}.

\begin{lemma} \label{lem:stab_subset_tau_inv}
We have the following equality of subspaces of $\Lie(B)$
\[
\stab_{\Lie(B)}(\tau)\subset\{\psi\in\Lie(B)\mid \tau(\pi_0(\psi))=\pi_0(\psi)\}.
\]
\end{lemma}
\begin{proof}
For $\psi\in\Lie(B)$, we compute
\begin{align}
\label{eq:stab_subset_lq_1}
\sigma_\psi^0\circ\tau(1)&=\sigma_\psi^0(1)=\pi_0\circ\sigma_\psi(1)=\pi_0\circ\ell_\psi(1)+\pi_0\circ\partial_\psi(1)=\pi_0(\psi),  \\
\label{eq:stab_subset_lq_2}
\tau\circ\sigma_\psi^0(1)&= \tau\circ\pi_0\circ\ell_\psi(1)+\tau\circ\pi_0\circ\partial_\psi(1)=\tau(\pi_0(\psi)).
\end{align}
For $\psi\in\stab_{\Lie(B)}(\tau)$, \eqref{eq:stab_subset_lq_1} and \eqref{eq:stab_subset_lq_2} agree by definition and hence we get the desired inclusion.
\end{proof}

\begin{theorem} \label{thm:stab_equal_lq} We have the following equality of bigraded subspaces of $\Lie(B)$
\[
\stab_{\Lie(B)}(\tau)=\lqq\oplus \Q b_0.
\]
\end{theorem}

\begin{proof} The space $\Lie(B)$ is bigraded by weight and depth, where $b_0$ is of bidegree $(1,0)$ and $b_i$ is of bidegree $(i,1)$ for $i\geq1$. We consider several cases depending on the bidegree $(m,n)$.

\begin{caselist}
\item $(m,n)=(1,0)$. We show that
\begin{align} \label{eq:stab_equal_lq_(1,0)}
\stab_{\Lie(B)}(\tau)[1,0]=\lqq[1,0]\oplus\Q b_0.
\end{align}
By definition of $\lqq$, we have $\lqq[1,0]=0$. On the other hand, we have $\stab_{\Lie(B)}(\tau)[1,0]\subset \Lie(B)[1,0]=\Q b_0$. We compute $\sigma_{b_0}(w)=\ell_{b_0}w+\partial_{b_0}(w)=b_0w+[w,b_0]=wb_0$. Thus, the projection of $\sigma_{b_0}(w)$ to $\QB^0$ is always zero,
\[\sigma_{b_0}^0\equiv 0.\]
In particular, $\sigma_{b_0}^0$ commutes with $\tau$ and thus $b_0\in \Lie(B)$ is contained in $\stab_{\Lie(B)}(\tau)$. We obtain equality \eqref{eq:stab_equal_lq_(1,0)}. 
\item $m\geq 2$ and $n\geq2$. We have that
\[\stab_{\Lie(B)}(\tau)[m,n]=\lqq[m,n].\]
This follows from Proposition \ref{prop:lq_subset_stab} and Lemma \ref{lem:stab_subset_tau_inv} together with the observation that $\lqq$ and $\{\psi\in\Lie(B)\mid \tau(\pi_0(\psi))=\pi_0(\psi)\}$ coincide in bidegree $(m,n)$.
\item $m \geq 2$ is odd and $n=1$. By the same argumentation as in Case 2, we get
\[
\stab_{\Lie(B)}(\tau)[m,1]=\lqq[m,1].
\]
\item $m \geq 2$ is even and $n=1$. We show that
\[
\stab_{\Lie(B)}(\tau)[m,1]=\lqq[m,1].
\]
By definition of $\lqq$, we have $\lqq[m,1]=0$. On the other hand, by construction we have $\stab_{\Lie(B)}(\tau)[m,1]\subset \Lie(B)[m,1]=\bigoplus_{k=1}^m\Q \cdot \ad_{b_0}^{m-k}(b_k)$. Let $\psi=\sum_{k=1}^m \lambda_k \ad_{b_0}^{m-k}(b_k) $ for some $\lambda_1,\ldots,\lambda_m \in \Q$, and assume that $\psi\in \stab_{\Lie(B)}(\tau)$. We show that this implies $\psi=0$. Then, we have $\stab_{\Lie(B)}(\tau)[m,1]=0$ and hence the claimed equality.

\noindent
We compute 
\begin{align} \label{eq:sigma_circ_tau}
&\sigma_{\psi}^0\circ\tau(b_1)=\sigma_\psi^0(b_1) \\
&=\sum_{k=1}^m \lambda_k\pi_0\Big(\ad_{b_0}^{m-k}(b_k) b_1+\sum_{l=1}^k(-1)^{k+l}\binom{k-1}{l-1}[b_{1+k-l},\ad_{b_0}^{m-k}(b_l)]\Big) \nonumber\\
&=\sum_{k=1}^m \lambda_k \Bigg(\sum_{n=0}^{m-k} (-1)^{m-k-n} \binom{m-k}{n} b_0^nb_kb_0^{m-k-n}b_1+\sum_{l=1}^k (-1)^{k+l} \binom{k-1}{l-1}b_{1+k-l}b_0^{m-k}b_l \nonumber\\
&\hspace{0.4cm} - \sum_{k=1}^l\sum_{n=0}^{m-k}(-1)^{m+l-n} \binom{k-1}{l-1}\binom{m-k}{n}b_0^nb_lb_0^{m-k-n}b_{1+k-l}\Bigg), \nonumber
\end{align}
and thus
\begin{align} \label{eq:tau_circ_sigma}
\tau\circ\sigma^0_\psi(b_1)&=\sum_{k=1}^m \lambda_k \Bigg(\sum_{n=0}^{m-k} (-1)^{m-k-n} \binom{m-k}{n} b_{m-k-n+1}b_0^{k-1}b_{n+1}\\
&\hspace{0.4cm}+\sum_{l=1}^k (-1)^{k+l} \binom{k-1}{l-1}b_0^{l-1}b_{m-k+1}b_0^{k-l}b_1 \nonumber\\
&\hspace{0.4cm} - \sum_{k=1}^l\sum_{n=0}^{m-k}(-1)^{m+l-n} \binom{k-1}{l-1}\binom{m-k}{n}b_0^{k-l}b_{m-k-n+1}b_0^{l-1}b_{n+1}\Bigg). \nonumber
\end{align}
As we assumed $\psi\in\stab_{\Lie(B)}(\tau)$, the two expressions \eqref{eq:sigma_circ_tau} and \eqref{eq:tau_circ_sigma} must coincide. In particular, also the coefficient of any word
\[b_0^{t_1}b_sb_0^{t_2}b_1,\qquad s+t_1+t_2=m,\ t_1\geq1,\]
must agree in both expressions. In \eqref{eq:sigma_circ_tau} this coefficient is given by
\begin{align} \label{eq:coeff_sigma_circ_tau}
 \lambda_s \Big( (-1)^{t_2}\binom{t_1+t_2}{t_1}- (-1)^{t_2}\binom{t_1+t_2}{t_1}\Big)=0,  
\end{align}
and in \eqref{eq:tau_circ_sigma} this coefficient equals
\begin{align} \label{eq:coeff_tau_circ_sigma}
\lambda_{t_1+t_2+1}\Big((-1)^{t_2} \binom{t_1+t_2}{t_1}-(-1)^{m-t_2-1}\binom{t_1+t_2}{t_1}\Big)&=\Big(1-(-1)^{m-1}\Big) \lambda_{t_1+t_2+1}(-1)^{t_2} \binom{t_1+t_2}{t_1} \nonumber\\
&=2\lambda_{t_1+t_2+1}(-1)^{t_2} \binom{t_1+t_2}{t_1} 
\end{align}
The last equality follows from $m$ being even. As \eqref{eq:coeff_sigma_circ_tau} and \eqref{eq:coeff_tau_circ_sigma} must coincide, we deduce that $\lambda_{t_1+t_2+1}=0$. As $t_1\geq1,t_2\geq 0$ were arbitrarily chosen, we get
\[\lambda_k=0 \quad \text{ for } k=2,\ldots,m,\]
and hence $\psi=\lambda_1 \ad_{b_0}^{m-1}(b_1)$. For $m\geq2$, the element $\pi_0(\ad_{b_0}^{m-1}(b_1))=b_0^{m-1}b_1$ is not $\tau$-invariant, hence we deduce from Lemma \ref{lem:stab_subset_tau_inv} that $\lambda_1=0$. Thus, we get $\psi=0$.
\end{caselist}
As the bigraded spaces $\stab_{\Lie(B)}(\tau)$ and $\lqq\oplus\Q b_0$ coincide on all components, we get the desired equality.
\end{proof}




\begin{proof}[Proof of Theorem \ref{thm:lq_Lie}] By Lemma \ref{lem:Ari_with_b0}, $\stab_{\Lie(B)}(\tau)[1,0]=\Q b_0$ is contained in the center of the Lie algebra $(\Lie(B),\{-,-\}_A)$. Thus, the projection
\[
\stab_{\Lie(B)}(\tau)\to\stab_{\Lie(B)}(\tau)[1,0]
\]
is a Lie algebra morphism where $\stab_{\Lie(B)}(\tau)[1,0]$ is considered as an abelian Lie algebra. By Theorem \ref{thm:stab_equal_lq}, we have that $\stab_{\Lie(B)}(\tau)=\lqq \oplus \Q b_0$, and hence 
\[
\ker(\stab_{\Lie(B)}(\tau)\to\stab_{\Lie(B)}(\tau)[1,0])=\lqq.
\]
We deduce that $\lqq$ is a Lie algebra as it is the kernel of a Lie algebra morphism. 
\end{proof}

\begin{corollary} We have the following inclusion of Lie subalgebras of $(\Lie(B),\{-,-\}_A)$
\[
\{\stab_{\Lie(B)}(\tau),\stab_{\Lie(B)}(\tau)\}_A\subset \lqq.
\]
\end{corollary}

\begin{proof} This follows directly from Theorem \ref{thm:stab_equal_lq} together with Lemma \ref{lem:Ari_with_b0}.
\end{proof}

\begin{remark}
There is a space $\mathfrak{bm}_0 \subset \QB$ explicitly given in \cite{Bu23}, which is conjecturally the graded dual of the indecomposables of $\qMZV/(\zeta_q(2),\zeta_q(4),\zeta_q(6))$. Hence, it is expected that $\mathfrak{bm}_0$ forms a Lie algebra and an explicit formula for a Lie bracket is also given in \cite{Bu23}. We expect that $\mathfrak{bm}_0$ also has the interpretation as a stabilizer of $\tau$ (with respect to a different Lie algebra action). We hope this can be used to give a proof that $\mathfrak{bm}_0$ is a Lie algebra.
\end{remark}

\section{Lie algebra injection between the stabilizers} \label{sec:comparison}

Define the injective $\Q$-algebra morphism
\[
    \theta_X : \QX \hookrightarrow \QB, \quad x_i \mapsto b_i, \text{ for } i \in \{0, 1\},
\]
and the injective $\Q$-algebra antimorphism
\[
    \theta_Y : \QY \hookrightarrow \QB^0, \quad y_n \mapsto b_n, \text{ for } n \in \mathbb{Z}_{>0}.
\]

Thanks to \cite[Theorem 7.10]{Bu25}, these maps induce an injective $\Q$-Lie algebra morphism
\begin{equation} \label{eq:def_theta}
\theta : (\ls, \{-, -\}) \hookrightarrow (\lqq, \{-, -\}_A), \quad \psi \mapsto \theta_X(\psi) + \theta_Y(\pi_Y(\psi)).
\end{equation}

Motivated by the inclusions $\ls \subset \stab_{\Lie(X)}(\Delta_Y)$ and $\lqq \subset \stab_{\Lie(B)}(\tau)$, we aim to construct an injective $\Q$-Lie algebra morphism $\stab_{\Lie(X)}(\Delta_Y) \hookrightarrow \stab_{\Lie(B)}(\tau)$ which extends $\theta : \ls \hookrightarrow \lqq$.

By Lemma \ref{lem:Ihara_with_x0_x1}, the space $\ls \oplus \Q x_1$ is also a $\Q$-Lie algebra for the Ihara bracket $\{-,-\}$.

\begin{proposition} \label{prop:theta_ex}
The map $\theta$ from \eqref{eq:def_theta} extends to an injective $\Q$-Lie algebra morphism
\begin{align*}
\theta^{(1)} : \big(\ls \oplus \Q x_1, \{-,-\}\big) & \to \big(\lqq, \{-,-\}_A \big), \\
\psi \oplus \lambda x_1 & \mapsto \theta(\psi) + \lambda b_1.
\end{align*}
\end{proposition}
\begin{proof}
Since $\theta : (\ls, \{-, -\}) \to (\lqq, \{-, -\}_A)$ is a Lie algebra morphism, and since $\{\psi,x_1\}=0$ for any $\psi\in\QX$, thanks to Lemma \ref{lem:Ihara_with_x0_x1}; to prove that $\theta^{(1)} : (\ls \oplus \Q x_1, \{-, -\}) \to (\lqq, \{-, -\}_A)$ is a Lie algebra morphism, it then suffices to verify that
\[
    \{\theta(\psi),b_1\}_A=0, \quad \forall \psi \in \ls.
\]
First, observe that by Remark \ref{rem:bracket_on_b0_b1} the Lie bracket $\{-,-\}_A$ equals the Ihara bracket on $\Q\langle b_0,b_1\rangle$. \newline
Fix some element $\psi \in \ls$. As $\theta_X(\psi)\in \Q\langle b_0,b_1\rangle$, we deduce from Lemma \ref{lem:Ihara_with_x0_x1} that
\begin{align} \label{eq:thetaX_x_1} 
\{\theta_X(\psi),b_1\}_A=0.
\end{align}
Since $(\lqq,\{-,-\}_A)$ is a $\Q$-Lie algebra and $b_1, \theta(\psi)\in \lqq$, we deduce from \eqref{eq:thetaX_x_1}
\begin{align} \label{eq:(thetaY(phi),b1)_in_lq}
\{\theta_Y(\pi_Y(\psi)), b_1\}_A = \{\theta(\psi), b_1\}_A \in \lqq.
\end{align}
Observe that since $\psi \in \Lie(X)$ we have $(\psi\mid x_1^n)=0$ for $n\geq2$. Hence, $\psi\in\ls$ satisfies 
\[
    0 = (\psi \mid x_1^n) = (\pi_Y(\psi) \mid y_1^n) = (\theta_Y(\pi_Y(\psi)) \mid b_1^n), \quad \forall n \geq 1.
\]
Thus,
\begin{equation} \label{claim:letter_bi}
    \text{each word in } \theta_Y(\pi_Y(\psi))\in \Q\langle b_1,b_2,\ldots\rangle \text{ contains at least one letter } b_i \text{ with } i>1.
\end{equation}
Thanks to Lemma \ref{lem:bracket_b1b2..}, the element $\{\theta_Y(\pi_Y(\psi)),b_1\}_A$ is also contained in $\Q\langle b_1,b_2,\dots\rangle$ and as the Lie bracket $\{-,-\}_A$ is depth-homogeneous, it follows from \eqref{claim:letter_bi} that
\begin{equation} \label{claim:bracket_letter_bi}
    \text{ each word in } \{\theta_Y(\pi_Y(\psi)),b_1\}_A \text{ contains at least one letter } b_i \text{ with } i>1.    
\end{equation}
Therefore, by definition of $\tau$, it follows from \eqref{claim:bracket_letter_bi} that
\begin{equation} \label{claim:tau_letter_b0}
    \text{each word in } \tau\big(\{\theta_Y(\pi_Y(\psi)), b_1\}_A\big) \text{ must contain the letter } b_0.    
\end{equation}
Thus from \eqref{claim:bracket_letter_bi} and \eqref{claim:tau_letter_b0} we deduce, for any word $w \in \QB$, that
\begin{equation} \label{eq:equaliff0}
    (\{\theta_Y(\pi_Y(\psi)),b_1\}_A \mid w) = (\tau\big(\{\theta_Y(\pi_Y(\psi)), b_1\}_A\big) \mid w) \Longrightarrow (\{\theta_Y(\pi_Y(\psi)),b_1\}_A \mid w) = 0.
\end{equation}
On the other hand, by \eqref{eq:(thetaY(phi),b1)_in_lq} and the definition of $\lqq$, $\{\theta_Y(\pi_Y(\psi)),b_1\}_A$ is $\tau$-invariant, meaning that the left hand side of \eqref{eq:equaliff0} is satisfied. It then follows from the aforementioned implication that
\[
    \{\theta_Y(\pi_Y(\psi)),b_1\}_A = 0,    
\]
as it is zero for any word in $\QB$. Hence
\[\{\theta(\psi),b_1\}_A=\{\theta_Y(\pi_Y(\psi)), b_1\}_A=0.
\]
Finally, since $\theta$ is injective and the spaces $\operatorname{im} \theta = \theta^{(1)}(\ls)$ and $\Q b_1 = \theta^{(1)}(\Q x_1)$ have a trivial intersection, we obtain the injectivity of $\theta^{(1)}$.
\end{proof}

\begin{theorem} \label{thm:theta:stabDeltaY_to_stabtau} The map $\theta$ from \eqref{eq:def_theta} induces an injective $\Q$-Lie algebra morphism between the stabilizers $\stab_{\Lie(X)}(\Delta_Y)$ and $\stab_{\Lie(B)}(\tau)$. Namely,
\begin{align*}
\theta^{(10)} : \big(\stab_{\Lie(X)}(\Delta_Y), \{-, -\}\big) & \to \big(\stab_{\Lie(B)}(\tau), \{-, -\}_A\big) \\
\psi \oplus \lambda_1 x_1\oplus \lambda_0 x_0 & \mapsto \theta(\psi) + \lambda_1 b_1 \oplus \lambda_0 b_0,
\end{align*}
where one uses the equalities $\stab_{\Lie(X)}(\Delta_Y) = \ls \oplus \Q x_1 \oplus \Q x_0$ from Theorem \ref{thm:stab_equals_ls} and $\stab_{\Lie(B)}(\tau) = \lqq \oplus \Q b_0$ from Theorem \ref{thm:stab_equal_lq}.
\end{theorem}

\begin{proof} By Proposition \ref{prop:theta_ex}, we have an injective Lie algebra morphism
\[\theta^{(1)} : \ls \oplus \Q x_1 \hookrightarrow \lqq, \quad \psi \oplus \lambda_1 x_1 \mapsto \theta(\psi) + \lambda_1 b_1.\]
Thanks to Lemmas \ref{lem:Ihara_with_x0_x1} and \ref{lem:Ari_with_b0}, it follows that $\{-,x_0\}$ and $\{-,b_0\}_A$ are zero maps. Therefore, 
the map $\theta^{(1)}$ extends to a Lie algebra morphism $\theta^{(10)}$ defined as above. \newline
Finally, since the spaces $\operatorname{im} \theta^{(1)}=\theta^{(10)}(\ls\oplus \Q x_1)$ and $\Q b_0=\theta^{(10)}(\Q x_0)$ have a trivial intersection, the injectivity of the map $\theta^{(10)}$ follows from the injectivity of $\theta^{(1)}$.
\end{proof}

\appendix 

\section{Review of Lie algebra actions}

\begin{definition}
    Let $(\mathfrak{g}, \langle-,-\rangle)$ be a $\Q$-Lie algebra and $V$ be a $\Q$-vector space.
    The Lie algebra $\mathfrak{g}$ acts on the space $V$ by endomorphisms if there is a $\Q$-Lie algebra morphism
    \[
        \mathrm{end} : (\mathfrak{g}, \langle-,-\rangle) \to \mathrm{End}_\Q(V), \quad \psi \mapsto \mathrm{end}_\psi
    \]
    where the space $\mathrm{End}_\Q(V)$ of endomorphisms of $V$ is equipped with the bracket
    \[
        [u, v] := u \circ v - v \circ u, \quad \forall u, v \in \mathrm{End}_\Q(V).
    \]
\end{definition}

\begin{propdef} \label{propdef:generic_stab}
    Let $(\mathfrak{g}, \langle-,-\rangle)$ be a $\Q$-Lie algebra acting on a $\Q$-vector space $V$ by endomorphisms $\mathrm{end} : (\mathfrak{g}, \langle-,-\rangle) \to \mathrm{End}_\Q(V)$.
    Let $v \in V$. The space
    \[
        \stab(v) := \{ \psi \in \mathfrak{g} \mid \mathrm{end}_\psi(v) = 0 \}
    \]
    is a Lie subalgebra of $\mathfrak{g}$ called the \emph{stabilizer Lie algebra} of the element $v$. 
\end{propdef}
\begin{proof}
    For $\psi, \phi \in \stab(v)$, one has
    \[
        \mathrm{end}_{\langle \psi,\phi\rangle}(v) = \mathrm{end}_\psi \circ \mathrm{end}_\phi(v) - \mathrm{end}_\phi \circ \mathrm{end}_\psi(v) = \mathrm{end}_\psi(0) - \mathrm{end}_\phi(0) = 0,
    \]
    thus proving that $\langle \psi,\phi\rangle \in \stab(v)$.
\end{proof}

For a $\Q$-algebra $A$, denote by $\mathrm{Der}_{\Q{\text-}\mathsf{alg}}(A)$ the space of derivations of algebras, that is, linear endomorphisms $\delta$ of $A$ satisfying
\[
    \delta(ab) = \delta(a) b + a \delta(b), \quad \forall a,b \in A.
\]
One immediately checks that $\mathrm{Der}_{\Q{\text-}\mathsf{alg}}(A)$ is a $\Q$-Lie algebra for the bracket
\[
    [\delta_1, \delta_2] := \delta_1 \circ \delta_2 - \delta_2 \circ \delta_1.
\]

\begin{definition}
    Let $(\mathfrak{g}, \langle-,-\rangle)$ be $\Q$-Lie algebra and $A$ be $\Q$-algebra. The Lie algebra $\mathfrak{g}$ acts on the algebra $A$ by derivations if there is a $\Q$-Lie algebra morphism
    \[
        \mathrm{der} : (\mathfrak{g}, \langle-,-\rangle) \to \mathrm{Der}_{\Q{\text-}\mathsf{alg}}(A), \quad \psi \mapsto \mathrm{der}_\psi.
    \]
\end{definition}

\begin{propdef}[{\cite[Proposition 2.2.]{elm}}]
Let $(\mathfrak{g}, \langle-,-\rangle)$ be a $\Q$-Lie algebra and $\triangleright : \mathfrak{g} \times \mathfrak{g} \to \mathfrak{g}$ be a bilinear map. Set
\begin{equation*} 
\langle \psi, \phi \rangle_\triangleright := \langle \psi, \phi \rangle + \psi \triangleright \phi - \phi \triangleright \psi,\qquad \psi,\phi\in\mathfrak{g}.
\end{equation*}
If for all $\psi,\phi,\lambda\in\mathfrak{g}$
\begin{enumerate}[label=(\roman*), leftmargin=*]
\item $\psi\ \triangleright$ is a derivation on $(\mathfrak{g}, \langle-,-\rangle)$, \label{post-lie1}
\item $\langle\psi,\phi\rangle_\triangleright \triangleright \lambda=(\psi\triangleright \phi-\phi\triangleright \psi)\triangleright \lambda$, \label{post-lie2}
\end{enumerate}
then $(\mathfrak{g},\langle-,-\rangle,\triangleright)$ is called a \emph{post-Lie algebra}. In this case, the pair $(\mathfrak{g}, \langle-,-\rangle_\triangleright)$ forms a Lie algebra. \label{propdef:post-Lie}
\end{propdef}

For a more comprehensive treatment of post‑Lie algebras, the reader may refer to §3 of \cite{BCK23}.

\begin{proposition} \label{prop:generic_act_end}
Let $(\mathfrak{g}, \langle-,-\rangle,\triangleright)$ be a post-Lie algebra. Denote by $U(\mathfrak{g})$ the universal enveloping $\Q$-algebra of $(\mathfrak{g}, \langle-,-\rangle)$. 
Then
\begin{enumerate}[label=(\alph*), leftmargin=*]
\item \label{generic_act_der} The $\Q$-Lie algebra $(\mathfrak{g}, \langle-,-\rangle_\triangleright)$ acts on the $\Q$-algebra $U(\mathfrak{g})$ by derivations
\[
\mathrm{der} : (\mathfrak{g}, \langle-,-\rangle_\triangleright) \to \mathrm{Der}_{\Q{\text-}\mathsf{alg}}(U(\mathfrak{g})), \quad \mathrm{der}_\psi := \psi \triangleright -,
\]
\item \label{generic_act_end} The $\Q$-Lie algebra $(\mathfrak{g}, \langle-,-\rangle_\triangleright)$ acts on the $\Q$-linear space $U(\mathfrak{g})$ by endomorphisms
\[
\mathrm{end} : (\mathfrak{g}, \langle-,-\rangle_\triangleright) \to \mathrm{End}_\Q(U(\mathfrak{g})), \quad \mathrm{end}_\psi := \ell_\psi + \mathrm{der}_\psi, 
\]
for any $\psi \in \mathfrak{g}$, where $\ell_\psi \in \mathrm{End}_\Q(U(\mathfrak{g}))$ is the left multiplication in the algebra $U(\mathfrak{g})$ by the element $\psi$. 
\end{enumerate}
\end{proposition}

\begin{proof}
\begin{enumerate}[label=(\alph*), leftmargin=*]
\item For any $\psi \in \mathfrak{g}$, we have $\mathrm{der}_\psi \in \mathrm{Der}_{\Q{\text-}\mathsf{alg}}$ thanks to condition \ref{post-lie1} in Proposition-Definition \ref{propdef:post-Lie}; and for $\psi, \phi \in \mathfrak{g}$, the identity
\[
\mathrm{der}_{\langle \psi, \phi\rangle} = [\mathrm{der}_\psi, \mathrm{der}_\phi]
\]
follows from condition \ref{post-lie2}. 
\item Let $\psi, \phi \in \mathfrak{g}$ and $w \in U(\mathfrak{g})$. We have
\begin{align*}
[\mathrm{end}_\psi, \mathrm{end}_\phi](w) & = \mathrm{end}_\psi \circ \mathrm{end}_\phi(w) - \mathrm{end}_\phi \circ \mathrm{end}_\psi(w) \\
& = (\ell_\psi + \mathrm{der}_\psi) \circ (\ell_\phi + \mathrm{der}_\phi)(w) - (\ell_\phi + \mathrm{der}_\phi) \circ (\ell_\psi + \mathrm{der}_\psi)(w) \\
& = \psi \phi w + \psi \mathrm{der}_\phi(w) + \mathrm{der}_\psi(\phi w) + \mathrm{der}_\psi \circ \mathrm{der}_\phi(w) \\
& \hspace{0.4cm} - \phi \psi w - \phi \mathrm{der}_\psi(w) - \mathrm{der}_\phi(\psi w) - \mathrm{der}_\phi \circ \mathrm{der}_\psi(w) \\
& = (\psi \phi - \phi \psi) w + (\mathrm{der}_\psi \circ \mathrm{der}_\phi - \mathrm{der}_\phi \circ \mathrm{der}_\psi)(w) + \psi \mathrm{der}_\phi(w) \\
& \hspace{0.4cm} + \phi \mathrm{der}_\psi(w) + \mathrm{der}_\psi(\phi) w - \phi \mathrm{der}_\psi(w) - \psi \mathrm{der}_\phi(w) - \mathrm{der}_\phi(\psi) w \\
& = (\psi \phi - \phi \psi + \mathrm{der}_\psi(\phi) - \mathrm{der}_\phi(\psi)) w + (\mathrm{der}_\psi \circ \mathrm{der}_\phi - \mathrm{der}_\phi \circ \mathrm{der}_\psi)(w) \\
& = (\langle \psi, \phi\rangle + \mathrm{der}_\psi(\phi) - \mathrm{der}_\phi(\psi)) w + [\mathrm{der}_\psi, \mathrm{der}_\phi](w) \\
& = \langle \psi, \phi\rangle_\triangleright w + \mathrm{der}_{\langle\psi, \phi\rangle_\triangleright}(w) = \mathrm{end}_{\langle\psi, \phi\rangle_\triangleright}(w),
\end{align*}
where the fourth equality follows from the fact that $\mathrm{der}_\psi$ and $\mathrm{der}_\phi$ are derivations; the sixth one from the fact that in $U(\mathfrak{g})$, we have $\langle \psi, \phi\rangle = \psi \phi - \phi \psi$; and the seventh one from \ref{generic_act_der}. \qedhere 
\end{enumerate}
\end{proof}

\bibliographystyle{amsalpha}
\bibliography{main}
\end{document}